\documentclass[review,onefignum,onetabnum]{siamart190516}

\newtheorem{thm}{Theorem}
\newtheorem{lem}{Lemma}
\newtheorem{cor}{Corollary}
\newtheorem{prp}{Proposition}

\newtheorem{rem}{Remark}

\newtheorem{ass}{assumption}

\newcommand{\eps}{\varepsilon}
\newcommand{\La}{\Lambda}
\newcommand{\la}{\lambda}

\def\cX{\mathcal{X}}
\def\cY{\mathcal{Y}}

\def\E{\mathbb{E}}

\def\wh{\widehat}
\newcommand{\be}{\begin{equation}}
\newcommand{\ee}{\end{equation}}
\newcommand{\br}{\begin{rem}}
\newcommand{\er}{\end{rem}}
\newcommand{\bq}{\begin{qu}}
\newcommand{\eq}{\end{qu}}
\newcommand{\bn}{\begin{enumerate}}
\newcommand{\en}{\end{enumerate}}
\newcommand{\bi}{\begin{itemize}}
\newcommand{\ei}{\end{itemize}}
\newcommand{\beas}{\begin{eqnarray*}}
\newcommand{\eeas}{\end{eqnarray*}}
\newcommand{\bea}{\begin{eqnarray}}
\newcommand{\eea}{\end{eqnarray}}

\newcommand{\dom}{\operatorname{dom}}

\newcommand{\N}{\mathbb{N}}

\newcommand{\R}{\mathbb{R}}
\DeclareMathOperator*{\argmin}{arg\,min}

\usepackage{lipsum}
\usepackage{graphicx}
\usepackage{stackengine}
\usepackage{amsfonts}
\usepackage{amssymb}
\usepackage{graphicx}
\usepackage{epstopdf}
\usepackage{algorithmic}
\usepackage{booktabs}
\usepackage{enumitem}
\usepackage{cleveref}
\ifpdf
  \DeclareGraphicsExtensions{.eps,.pdf,.png,.jpg}
\else
  \DeclareGraphicsExtensions{.eps}
\fi

\newsiamremark{remark}{Remark}
\newsiamremark{example}{Example}
\newsiamremark{hypothesis}{Hypothesis}
\crefname{hypothesis}{Hypothesis}{Hypotheses}
\newsiamthm{claim}{Claim}
\newsiamthm{prop}{Proposition}
\newsiamthm{defi}{Definition}

\headers{Learning the Optimal Regularization Parameter}{J. Chirinos-Rodríguez et al.}

\title{On Learning the Optimal Regularization Parameter in Inverse Problems}

\author{Jonathan Chirinos-Rodríguez\thanks{MaLGa, DIMA, Dipartimento di Eccellenza 2023-2027, Università degli Studi di Genova, Genoa, Italy (\email{c.rodriguez@dima.unige.it}, \email{molinari@dima.unige.it}, \email{ernesto.devito@unige.it}, \email{silvia.villa@unige.it})} \and Ernesto De Vito\footnotemark[1] \and Cesare Molinari\footnotemark[1] \and Lorenzo Rosasco\thanks{MaLGa, DIBRIS, Università degli Studi di Genova, Genoa, Italy \& Center for Brains, Minds and Machines, MIT, Cambridge, USA \& Istituto Italiano di Tecnologia, Genoa, Italy (\email{lorenzo.rosasco@unige.it}) } \and Silvia Villa\footnotemark[1]}

\usepackage{amsopn}

\ifpdf
\hypersetup{
  pdftitle={On Learning the Optimal Regularization Parameter in Inverse Problems}
  pdfauthor={J. Chirinos-Rodríguez, E. De Vito, C. Molinari, L. Rosasco, S. Villa}
}
\fi

\begin{document}

\maketitle

\begin{abstract}
	Selecting the best regularization parameter in inverse problems is a classical and yet challenging problem. Recently, data-driven approaches have become popular to tackle this challenge. These approaches are appealing since they do require less a priori knowledge, but their theoretical analysis is limited. In this paper, we propose and study a statistical machine learning approach, based on empirical risk minimization. Our main contribution is a theoretical analysis, showing that, provided with enough data, this approach can reach sharp rates while being essentially adaptive to the noise and smoothness of the problem. Numerical simulations corroborate and illustrate the theoretical findings. Our results are a step towards grounding theoretically data-driven approaches to inverse problems. 
\end{abstract}



\tableofcontents

\section{Introduction} \label{back}
Let $(\mathcal{X}, \, \langle \cdot, \cdot\rangle_{\mathcal{X}})$ and $(\mathcal{Y}, \, \langle \cdot, \cdot\rangle_{\mathcal{Y}})$ be real  separable Hilbert spaces
and $A\colon\mathcal{X}\to\mathcal{Y}$ a forward
operator.
Given $A$ and a datum $y\in \cY$, the corresponding inverse problem is to find $x^*\in \cX$ solving
$$
A(x^*)= y.
$$
In practice, only perturbed data are typically available, that is 
$$
\wh y = y+\eps, \quad \quad \|\eps\|_\cY\le \tau,
$$
where we considered a deterministic noise model.
The above problem is often ill-posed and, in particular, solutions might not depend smoothly on the data. Regularization theory provides a principled approach towards finding stable solutions, see e.g. \cite{burgerbenning,engl}. 
First, a family of regularization operators is defined for every $\la\in(0 +\infty)$: $f_\la:\cY\to \cX$. Then, a choice is specified for the regularization parameter  $\la$. 
Ideally, for some given discrepancy $\ell$, such a choice should allow to optimally control the error $\ell(f_\la(\wh y ),  x^*)$. Classical strategies for choosing the regularization parameter 
are divided in \emph{a priori}, where $\la=\la(\tau, x^*)$ and \emph{a posteriori}, where $\la=\la(\tau)$.
A priori choices are primarily of theoretical interest. The reason is that they allow to derive sharp error estimates that can be shown to match corresponding lower bounds, see e.g. \cite{engl}.
However, they are usually impractical since they depend on the unknown solution $x^*$ -- or rather on its regularity properties expressed by some smoothness parameters.
A posteriori choices, such as the classic Morozov discrepancy principle \cite{moroz}
are adaptive to the knowledge of the regularity properties of $x^*$, but still require the noise level $\tau$. Since in many practical scenarios this information might not be available, a number of alternative strategies have been proposed, including generalized cross-validation \cite{golub,wahba},  quasi-optimality criterion \cite{baure,tiar},  L-curve method \cite{hansen}, and methods based on an estimation of the mean squared error, see e.g. \cite{sugar} and references therein. 

In recent years, \emph{data-driven} approaches to inverse problems have received much attention since they seem to provide improved results, while circumventing some limitations of classical approaches, see \cite{arridge} and references therein. The starting point of data-driven approaches is the assumption that a finite set of pairs of data and exact solutions $(\wh y_1, x^*_1),\dots, (\wh y_n, x^*_n)$ is available. This {\em training set} can be used to define, or refine, a regularization strategy to be used on any future datum $\wh y$ for which an exact solution is not known. This perspective has been already considered to provably learn a regularization parameter choice. For example, in \cite{ratti2021} a general approach is analyzed to learn a regularizer in Tikhonov-like regularization schemes for linear inverse problems. Indeed,  these results can be adapted to learn the best regularization parameter in some cases. Another learning approach is analyzed in \cite{devfornau} and \cite{naukeret}, where an unsupervised approach is studied. A bilevel optimization perspective is taken in \cite{sav18}, where some theoretical results are also given.

In this paper, we consider one of the most classical machine learning approaches, namely empirical risk minimization (ERM). We study the regularization parameter choice defined by the following problem,
$$
\min_{\la\in \Lambda} \frac 1 n \sum_{i=1}^n\ell(f_\la(\wh y_i ), x^*_i)
$$
where  $\Lambda $ is a suitable finite set of candidate values for $\la$.
Our main contribution is characterizing the error performance of the above approach.
Towards this end, we consider a statistical inverse problems framework 
and tackle the question with the aid of tools from statistical learning theory \cite{cucksma,vapnik}. The theory of ERM is well established, and the class of functions we need to consider is parameterized by just one parameter-- the regularization parameter. However, the dependence on such a parameter is nonlinear/nonsmooth and possibly hard to characterize, making the application of standard ERM results not straightforward. To circumvent this challenge we borrow ideas from the literature of model selection in statistics and machine learning \cite{devroye,gyorfi} and in particular, we adapt ideas from \cite{capon}. Our theoretical analysis shows that the ERM approach for learning the best regularization parameter can essentially achieve the same performance of an ideal a-priori choice. As we will see, this is true up to an error term, which decreases fast with the size of the training set. General results are illustrated considering several inverse problems scenarios. In particular, we discuss the case of linear inverse problems with spectral regularization methods and Tikhonov regularization with general convex regularizers in Sections \ref{sec:spectral} and \ref{sec:convex} respectively. Also, we consider non-linear inverse problems in Hilbert spaces and the corresponding Tikhonov regularization in Section \ref{sec:nonlin}. The theoretical results are illustrated through numerical experiments in Section \ref{expers} for spectral regularization methods and sparsity promoting norms.

\section*{Notation}
In the following, we assume that $(\Omega, P)$ is a probability space. Random variables will be denoted in capital letters. Given an element $x$ in a Hilbert space $(\cX, \langle \cdot, \cdot\rangle_\cX)$, $\|x\|_\cX$ denotes the corresponding norm, i.e. $\|x\|_\cX=\sqrt{\langle x,\, x\rangle_{\mathcal{X}}}$. Moreover, if $(\cY, \langle \cdot, \cdot\rangle_\cY)$ is also a Hilbert space, given a linear and bounded operator $A\colon\cX\to\cY$, we denote by $A^*$ its adjoint operator and, if $A$ is injective,  by $A^{-1}$ its inverse. With $\|\cdot\|_{\mathrm{op}}$ we denote the operator norm. Finally, the subdifferential of a proper, convex and lower semicontinuous function $f\colon \cX \mapsto \mathbb{R}\cup \{+\infty\}$ is the set-valued operator $\partial f\colon   \cX\to 2^{\cX}$ defined by
$$
	x\mapsto\{u\in \cX \ | \ \text{for every }
	y\in \cX, \ f(x)+\langle y-x, \
	u\rangle_{\cX}\leq f(y)\}.
$$

\section{Learning one parameter functions}\label{sec:learn1param}
In this section, we derive statistical learning results to learn functions parameterized by one parameter. In particular, in the context of learning in inverse problems, this will be the regularization parameter. For the time being, we consider an abstract learning framework. 

Let $(Y, X)$ be a pair of random variables with values in $\cY\times \cX$ and let  $(Y_i, X_i)_{i=1}^n$ be  $n$ identical and independent copies of $(Y,X)$.  For $\la\in(0,+\infty)$, let $f_\la:\cY\to \cX$ be a family of measurable functions parametrized by $\la$. Given a measurable  loss function $\ell:\cX\times \cX\to [0, +\infty)$, for all measurable functions $f:\cY\to \cX$ consider the expected risk
$$
L(f)= \E[\ell(f(Y), X)]. 
$$
and the empirical risk
$$
\wh L (f)= \frac 1 n \sum_{i=1}^n \ell(f(Y_i), X_i).
$$
Moreover, for some $N \in \N$, define $\Lambda$, the finite grid of regularization parameters, as
\begin{equation}\label{eq:Lambda}
 \Lambda= \{\la_1, \dots, \la_N\}   
\end{equation}
with $0<\la_1\le \la_2 \dots \le \la_N<\infty$. Considering  the empirical risk minimization (ERM), we let
\begin{equation}\label{eq:lambdahat}
\wh \la _\Lambda \in \argmin_{\la \in \Lambda} \wh  L(f_\la).    
\end{equation}
We aim at characterizing $L (f_{\wh \la _\Lambda})$,  namely the expected risk corresponding to the regularization parameter chosen accordingly to the rule in \eqref{eq:lambdahat}. An idea would be to compare it directly to $\min_{\la\in(0,+\infty)} L(f_\la)$. Instead, as discussed next, we assume that a suitable error bound $\min_\la L(f_\la)\le U(\la_*)$ is available, and then we compare $L(f_{\wh{\la}_\Lambda})$ to $U(\lambda_*)$. Next, we list and comment the main assumptions.
\begin{ass}\label{ass:ell}
The loss function $\ell$ is bounded by a constant $M>0$.
\end{ass}
In the following, we will consider loss functions defined by classic discrepancy errors in inverse problems. In particular, we focus on Hilbertian norms, see Sections \ref{sec:spectral} and \ref{sec:nonlin}, and Bregman divergences associated with convex functionals, see Section \ref{sec:convex}. While none one of these examples are bounded, since we will assume $X$ to be almost surely bounded, a bounded loss will be obtained by composing the discrepancy with suitable truncation operators.
\begin{ass}\label{abound}
There exists $U:(0,+\infty)\to (0,+\infty)$ such that, for every $\lambda\in (0,+\infty)$,
\be\label{bound}
L(f_\la)\le U(\la).
\ee
Moreover, there exists $\la_*>0$ such that 
\be\label{las}
\la_*\in  \argmin_{\la \in (0,+\infty)} U(\la).
\ee
Finally,  there exists a non decreasing  function $C:[1,+\infty)\to [0,+\infty)$ such that, for all $q\geq 1$,
\be\label{smooth_bound}
 U(q\la_*) \le C(q) U(\la_*).
\ee
\end{ass}
The main reason for the above assumption is to avoid smoothness conditions on the dependence of $f_\la$ on $\la$ which are required in classic studies of ERM, see e.g. \cite{cucksma}. This assumption might seem unusual for a learning setting but, as shown in Sections \ref{sec:spectral}, \ref{sec:nonlin} and \ref{sec:convex}, it is naturally satisfied in the context of inverse problems. Moreover, this is the usual strategy to design a priori choices of the regularization parameter, since in this latter setting it is often possible to derive tight bounds,  
in the sense that the two quantities, $L(f_\la)$ and $U(\la)$, have the same behavior with respect to $\lambda$ and the noise level, and therefore
$\min_{\la\in(0,+\infty)} L(f_\la)$ is comparable to $U(\lambda_*)$ (see e.g. \cite[Chapter 4]{engl}). 
  We make one last assumption on how {\em large} is the set of candidate values $\Lambda$. 
\begin{ass}\label{aroughsnr}
Let $\Lambda$ be defined as in \eqref{eq:Lambda}. Assume that 
\be\label{approx_ass}
\la_*\in[\la_1, \la_N]
\ee
and, for every $j=1, \dots, N$, $\la_j=\la_1 Q^{j-1}$, where
\be\label{Q}
Q=\displaystyle\left(\frac{\la_N}{\la_1} \right)^{\frac{1}{N-1}}.
\ee
\end{ass}
The above assumption states that we can choose a sufficiently large interval for our discretization so that the optimal regularization parameter $\la_*$ in \eqref{las} always falls within the interval. This is an approximation assumption which is satisfied in practice by taking $\la_1$ sufficiently small (and $\la_N$ sufficiently big).

Given the above assumptions,  we next show that the choice  $\wh \la _\Lambda$ achieves an error close to that of  $\la_*$.
 
\begin{thm}\label{genthm}
Let Assumptions~\ref{ass:ell}, \ref{abound} and \ref{aroughsnr} be satisfied and let $\eta\in(0,1)$. Then, with probability at least $1- \eta$,
$$
L(f_{\wh \la _\Lambda})\le 2 C(Q) U(\la_*) + \frac{13M}{2n} \log\frac{2 N}{\eta}.
$$
\end{thm}
The above result shows that $\wh \la_\Lambda$ achieves an error of the same order of $\la_*$ up to a multiplicative factor depending on $C(Q)$ and a corrective term which decreases as  $1/n$. 

From the expression~\eqref{Q}, once the minimal and maximal elements of the discretization are fixed, we can see that $Q\approx 1$ if $N$ is large enough. At the same time, taking $N$ large has a minor effect on the bound, since the corrective term depends logarithmically on $N$.  In the following, we provide concrete examples in the context of inverse problems that illustrate and instantiate the above results.   

We first provide the proof of Theorem~\ref{genthm}.

\subsection{Proof of Theorem 1}
We begin providing a sketch of the main steps in the proof.
The idea  is to first compare the behavior of $\wh \la _\Lambda$ to that of 
$$
\la_\Lambda \in \argmin_{\la\in \Lambda} L(f_\la),
$$
which is the ideal regularization parameter choice when restricting the search to $\Lambda$.
Indeed, we prove in Lemma~\ref{unionbound} that with high probability
$$
L(f_{\wh \la _\Lambda})\leq  2 L(f_{\la _\Lambda}) + c\frac{\log(2N)}{n},
$$
for some constant $c>0$.
Then, in  Lemma~\ref{approxlemma} we show that there exists  $1\le q < Q$ such that $q\la_*\in\La$ and so
$$
L(f_{\la_\Lambda})\le L(f_ {q \la_*}).
$$
Combining the above results and using condition~\eqref{smooth_bound}, we get with high probability that 
$$
L(f_{\wh \la _\Lambda})\lesssim 2L(f_{q \la_*}) + \frac{\log(2N)}{n}\lesssim 
2C(Q) U(\la_*) + \frac{\log(2N)}{n},
$$
which is the desired result. We next provide the detailed proof. First, we introduce the following probabilistic lemma.
\begin{lem}\label{unionbound} 
Under Assumption~\ref{ass:ell}, for $\eta\in(0,1)$ we have that, with probability at least $1- \eta$,
$$
L(f_{\wh \la _\Lambda})\le 2L(f_{\la_\Lambda})  + \frac{13M}{2n} \log\frac{2 N}{\eta}.
$$
\end{lem}
The proof is based on a classic union bound argument and the following concentration inequality, see Proposition $11$ in  \cite{capon}, which we report for simplicity.
\begin{prp}\label{benny}
Let  $Z_1, \dots Z_n$ be a sequence of i.i.d. real random variables with mean $\mu$, such that $|Z_i|\le B$ a.s. and $\E[|Z_i-\mu|^2]\le \sigma^2$. Then for all $\alpha, \eps>0$
\be\label{conc}
\displaystyle P\left\{ \left|\frac 1 n \sum_{i=1}^n Z_i - \mu\right| \ge \eps+\alpha \sigma^2\right\}\le 2 e^{-\frac{6n\alpha\eps}{3+4\alpha B}}.
\ee
\end{prp}
The idea of the proof is adapted from \cite{capon}. 
\begin{proof}(of Lemma~\ref{unionbound}).
For $\la \in \Lambda$ , let $Z_i(\la)=\ell(f_\la(Y_i), X_i)$, $i=1,...,n$. Then, 
$$
\frac 1 n \sum_{i=1}^n Z_i(\la)=\wh L (f_\la),
$$
and 
$$
\E[Z_i(\la)]= L (f_\la).
$$
Moreover, since the loss is bounded by Assumption~\ref{ass:ell}, then $Z_i(\la)\le M$ and this implies
$$
\E[|Z_i(\la)|^2]=  \E[\ell(f_\la(Y_i), X_i)\ell(f_\la(Y_i), X_i)]\le M L(f_\la).
$$
Now, we apply~\eqref{conc} with $B=M$ and, by recalling that  $\mathbb{E}[|Z_i(\lambda)-\E[Z_i(\la)]|^2]\leq \mathbb{E}[|Z_i(\lambda)|^2]$,  we fix $\sigma^2= M L(f_\la)$. We then get, for each $\la\in \Lambda$ and for all $\alpha, \eps>0$, 
$$
 P\left\{ |\wh L (f_\la) - L (f_\la)| \ge \eps+\alpha M L (f_\la)\right\}\le 2 e^{-\frac{6n\alpha\eps}{3+4\alpha M}}. 
$$
Moreover, since the probability of a union of events is less or equal than the sum of their probabilities, we have that, for all $\alpha, \eps>0$,
$$
 P\left( \bigcup_{\la\in \Lambda} \left\{ |\wh L (f_\la) - L (f_\la)| \ge \eps+ \alpha M L (f_\la)\right\}\right)\le 2|\Lambda| e^{-\frac{6n\alpha\eps}{3+4\alpha M}}. 
$$
Now let $\eta\in(0,1)$. Since the above is valid for any $\alpha>0$, fix $\alpha=1/(3M)$. With this choice, let $\eps=\frac{13M}{6n}\log\frac{2|\Lambda|}{\eta}$. Then, with probability at least $1-\eta$, for all $\la\in \Lambda$ we have that
$$
\wh L (f_\la) \le \frac{4}{3}L (f_\la)+\eps
$$
and 
$$
 L (f_\la)\le\frac{3}{2}\left( \wh L (f_\la) +\eps\right).
$$
Using the above inequalities and the definition of $\wh \la_\La$ we have that,
\beas
L( f_{ \wh \la_\Lambda})
&\le& \frac{3}{2}\left( \wh L (f_{\wh \la_\Lambda}) +\eps\right)\\
&\le& \frac{3}{2}\left( \wh L (f_{ \la_\Lambda}) +\eps\right)\\
&\le& 2 L (f_{ \la_\Lambda}) +3\eps.
\eeas
The result follows by plugging in the expression of $\eps$ and by recalling that $|\Lambda|=N$.
\end{proof}
Note that the above result holds under minimal assumptions. Indeed, the structural assumptions we introduced are used to prove the following lemma. 
\begin{lem}\label{approxlemma}
Let Assumptions \ref{abound} and \ref{aroughsnr} be satisfied and consider $\la_*$ as in Assumption \ref{abound}. Then, there exists $1\le q \le Q$ such  that $q\la_*\in\Lambda$ and so
$$L(f_{\la_\Lambda})\le L(f_ {q \la_*}).$$ 
\end{lem}
\begin{proof}
From Assumption~\ref{aroughsnr}, 
since $\la_*\in [\la_1, \la_N]$, there exists $j_0 \in \{2, \dots, N\}$ such that 
$$\la_{j_0-1}\le\la_*\le  \la_{j_0}.$$ 
If we let $q=\la_{j_0}/\la_*$, then $q\la_*=\la_{j_0}\in \Lambda$. It is only left to prove that $1\le q\le Q$. Given the definition of $Q$ and the construction of $\Lambda$, if we divide the above inequalities by $\la_{j_0}$, then 
$$
\frac{1}{Q}\le \frac{1}{q}\le 1,
$$
so that 
$$
1\le q \le Q.
$$
Finally, by the definition of $\la_\Lambda$, we get 
$$
L(f_{\la_\Lambda})\le L(f_ {q \la_*}),
$$
concluding the proof.
\end{proof}

We add one final remark. 
  \begin{rem}[Comparison with union bound combined with Hoeffding]\label{rem:excessrisk}
A slightly different estimate can be obtained using a union bound argument and a different concentration result, namely Hoeffiding inequality~\eqref{hoeff_conc}. Indeed, if we let $\eta\in(0,1)$, the following bound holds with probability 
at least $1-\eta$:
\be\label{hoeff}
L(f_{\wh \la _\Lambda})\le L(f_{\la_\Lambda})  + 2 \sqrt{\frac{M}{n} \log\frac{2 N}{\eta}}.
\ee
 Compared to the estimate obtained in Lemma~\ref{unionbound}, the above inequality avoids the factor $2$ in front of $ L(f_{\la_\Lambda})$. However, the dependence on the data cardinality $n$ is considerably worse. By using inequality~\eqref{hoeff} in place of Lemma~\ref{unionbound}, it is possible to derive a result analogous to Theorem~\ref{genthm}. Again, this allows to improve the bound by a factor of $2$ while achieving a much worse dependence on the number of data points. For completeness, we report the proof of inequality~\eqref{hoeff}, which is based on Hoeffding's inequality:
\be\label{hoeff_conc}
\displaystyle P\left\{ \left|\frac 1 n \sum_{i=1}^n Z_i - \mu\right| \ge \eps\right\}\le2 e^{-\frac{n\eps^2}{2B}},
\ee
where $B$ is an upper bound on the random variables $Z_i$, as in Proposition~\ref{benny}. Indeed, by adding the subtracting the empirical risks we have that,  
\begin{align*}
L(f_{\wh \la_\Lambda})- L(f_{ \la_\Lambda})&=
L(f_{\wh \la_\Lambda})-\wh L(f_{\wh \la_\Lambda})+
\wh L(f_{\wh \la_\Lambda})- \wh L(f_{ \la_\Lambda})+
\wh L(f_{ \la_\Lambda})- L(f_{ \la_\Lambda})\\
&\le L(f_{\wh \la_\Lambda})-\wh L(f_{\wh \la_\Lambda})+
\wh L(f_{ \la_\Lambda})- L(f_{ \la_\Lambda})\\
&\le 2 \sup_{\la \in \Lambda} |L(f_{\la})- \wh L(f_{\la})|,
\end{align*}
using the fact that the term $\wh L(f_{\wh \la_\Lambda})- \wh L(f_{ \la_\Lambda})$ is negative by definition of $\wh \la_\Lambda$. Then, combining~\eqref{hoeff_conc} and a union bound, we get
$$
 P\left\{ \sup_{\la \in \Lambda} |L(f_{\la})- \wh L(f_{\la})| \ge \eps\right\}\le 2 N e^{-\frac{n\eps^2}{M}}. 
$$
Inequality~\eqref{hoeff} follows by setting $\eta= 2 N e^{-(n\eps^2)/M}$ and deriving the expression for $\eps$.
\end{rem}

\section{Spectral regularization for linear inverse problems}\label{sec:spectral}
In this section, we illustrate the general results considering spectral regularization methods for a class of stochastic linear inverse problems, extending the classical deterministic framework. 
The key point is to derive a suitable error bound and a corresponding a priori parameter choice so that Assumption \ref{abound} holds.
Let $\cX, \cY$ be real and separable Hilbert spaces, let $A\colon \cX\to\cY$ be a linear and bounded operator and assume that $\|A\|_{\mathrm{op}}\le 1$. Then, let $X,\eps$ be a pair of random variables with values in $\cX$ and $\cY$ respectively, and  
\begin{equation}\label{eq:linearip}
Y= A X+\eps, \quad \text{a.s.} 
\end{equation}
We make several assumptions. The first is on the noise $\eps$.
\begin{ass}\label{lip_noise}
We assume that 
$$
\E[\eps|X]=0
$$
and, moreover, that there exists $\tau>0$ such that
$$
\E[\|\eps\|^2_\cY|X]\le \tau^2.
$$
\end{ass}
The above condition is a simple and natural stochastic extension of the classical bounded variance assumption. We also assume that $X$ satisfies the following stochastic extension of the classical H\"older source conditions \cite{engl}. 
\begin{ass}\label{lip_source}
The random variable $X$ is such that $\|X\|_{\mathcal{X}}\leq 1$ a.s. and there exist a random variable $Z$ with values in $\cX$,  and $\beta$, $s>0$ such that,
$$
X= (A^*A)^s Z,
$$
and 
$$
\E[\|Z\|^2_\cX]\le \beta^2.
$$
\end{ass}
In this setting, a corresponding Tikhonov regularized estimator is defined as
\be\label{eq:tykhonov}
X_\la= \argmin_{x\in \cX}  \|Ax- Y\|_\cY^2+\la\|x\|^2_\cX.
\ee
Clearly, $X_\lambda=X_\lambda(Y)$, but we omit the dependence for conciseness. A more explicit expression is given by 
\begin{equation}\label{lineartikhonov}
X_\la= (A^*A+\la I)^{-1}A^*Y.
\end{equation}
More generally, the class of spectral regularization methods is given by 
\begin{equation}\label{eq:specfunctions}
X_\la= g_\la(A^*A)A^*Y,
\end{equation}
defined by a suitable function $g_\la:(0,1] \to \R$ using spectral calculus. Note that the above expression ensures that $X_\la$ is measurable, since it is the image of a linear operator applied to $Y$. 

The following assumption characterizes the key properties required on $g_\la$.
\begin{ass}\label{lip_spectral_reg}
There exists a constant $C_1>0$ such that,
for all $\la\in(0, +\infty)$, 
$$
\sup_{\sigma\in(0,1]}|g_\la(\sigma)\sqrt{\sigma}|\le\frac{C_1}{\sqrt{\la}}.
$$
Moreover, there is a constant $C_2>0$ and $\alpha>0$ such that, for $s>0$ as in Assumption~\ref{lip_source},
\begin{equation}
    \label{eq:qual}
\sup_{\sigma\in(0,1]}|(1-g_\la(\sigma)\sigma)\sigma^s|\le  C_2 \la ^{\alpha}.
\end{equation}
\end{ass}
Assumption \ref{lip_spectral_reg} is satisfied by a large class of filter functions such as Tikhonov regularization, the Landweber iteration, that is gradient descent on the least squares error, spectral cut-off, heavy-ball methods and the $\nu$-method \cite{engl}, or Nesterov acceleration \cite{neubauer}. We add some remarks regarding this assumption.

Note that the first assumption implies that the norm of the regularization operator $g_\la(A^*A)A^*$ is always bounded and controlled by $\la$. The second is an approximation condition, which characterizes the extent to which the considered spectral regularization method can take advantage of the regularity of the problem, expressed by the source condition. For many spectral regularization methods, there is $q>0$ such that
\[
 \sup_{\sigma\in(0,1]} |(1-g_\la(\sigma)\sigma)\sigma^\nu|\le  C_2 \la ^{\nu}, \quad \text{for every } \ \nu\leq q.
\]
The number $q$ is called qualification parameter and depends on the regularization method $g_\la$; see \cite{rosbau}. Therefore, Assumption~\ref{lip_spectral_reg} is satisfied for $\alpha=\min(q, s)$.  Both of the above assumptions allow us to derive suitable error bounds and corresponding a priori regularization parameter choice, extending classical results in the deterministic setting.

\begin{thm}\label{thm:apriori_spectral}
Under Assumptions~\ref{lip_noise},~\ref{lip_source} and~\ref{lip_spectral_reg}, the following bound holds for all $\la\in(0, +\infty)$, 
\begin{equation}\label{apriori_spectral}
\E[\|X_\la- X\|_\cX^2]\le  C_1^2\frac{\tau^2}{\la}+ C_2^2 \beta^2\la^{2\alpha}.
\end{equation}
In particular, taking 
$$
\la_*=\left(\frac{C_1^2}{2\alpha C_2^2}\right)^{1/(2\alpha+1)} \left(\frac{\tau}{\beta}\right)^{2/(2\alpha+1)},
$$ 
the following bound holds
\begin{equation}\label{final_bound}
    \E[\|X_{\la_*}- X\|_\cX^2]\le (2\alpha+1)\left[\left(\frac{C_1^2}{2\alpha}\right)^{2\alpha}C_2^2\right]^{1/(2\alpha+1)} \left(\frac{\tau^{2\alpha}}{\beta}\right)^{2/(2\alpha+1)}.
\end{equation}
\end{thm}

\begin{proof}
To relate $X_\la$ and $X$, we observe that
$$
\E[X_\la|X]=\E[g_\la(A^*A)A^*Y|X]=\E[g_\la(A^*A)A^*AX|X]=g_\la(A^*A)A^*AX,
$$
where we used the definition of $Y$ and Assumption~\ref{lip_noise}.
Then, we can decompose the deviation of $X_\la$ to $X$ as
\bea
X_\la-X&=&  X_\la-\E[X_\la|X]+\E[X_\la|X]-X\nonumber\\
&=& g_\la(A^*A)A^*(Y-AX)+ (g_\la(A^*A)A^*A-I)X\nonumber\\
&=& g_\la(A^*A)A^*\eps+ (g_\la(A^*A)A^*A-I)(A^*A)^s Z\label{dec2}.
\eea
Next, recall that, under Assumption~\ref{lip_spectral_reg}, the following operator estimates hold
\be\label{op_est}
\|g_\la(A^*A)A^*\|_{\mathrm{op}}\le \frac{C_1}{\sqrt{\la}},\quad \|(I-g_\la(A^*A)A^*A)(A^*A)^s\|_{\mathrm{op}}\le C_2 \la^{\alpha},
\ee
see e.g. \cite{engl}. If we take the expectation of the squared norm in~\eqref{dec2} and develop the square, we get
$$
\E[\|X_\la-X\|_\cX^2]= \E[\|g_\la(A^*A)A^*\eps\|^2_\cY]
+\E[\|(g_\la(A^*A)A^*A-I)X\|^2_\cX], 
$$
since, by Assumption ~\ref{lip_noise}, we have
$$
\begin{aligned}
\E&[\langle g_\la(A^*A)A^* \eps, (g_\la(A^*A)A^*A-I)X\rangle_\cX]\\
&=\E[\langle g_\la(A^*A)A^*\E[\eps| X],(g_\la(A^*A)A^*A-I)X\rangle]=0.
\end{aligned}
$$
Then, using again Assumptions~\ref{lip_noise}, \ref{lip_source}, and \ref{lip_spectral_reg} as well as the estimates~\eqref{op_est}, we derive 
$$
\begin{aligned}
\E[\|X_\la-X\|_\cX^2]&\le \|g_\la(A^*A)A^*\|_{\mathrm{op}}^2\E[\|\eps\|^2_\cY]+
\|(I-g_\la(A^*A)A^*A)(A^*A)^{s}\|_{\mathrm{op}}^2\E[\|Z\|^2_\cX]\\
&\le 
C_1^2\frac{\tau^2}{\la}+ C_2^2 \beta^2\la^{2\alpha}.
\end{aligned}
$$
Finally, the value of $\la$ minimizing the above bound is 
$$
\la_*=\left(\frac{C_1^2\tau^2}{2\alpha C_2^2\beta^2}\right)^{1/(2\alpha+1)},
$$
and the corresponding error bound is 
$$
\E[\|X_{\la_*}- X\|_\cX^2]\le (2\alpha+1)\left[\left(\frac{C_1^2}{2\alpha}\right)^{2\alpha}C_2^2\right]^{1/(2\alpha+1)} \left(\frac{\tau^{2\alpha}}{\beta}\right)^{2/(2\alpha+1)},
$$
which is the inequality that we were aiming for.
\end{proof}
The expression given in \eqref{apriori_spectral} provides a bound, for any value of the regularization parameter, of the distance between the regularized and the exact solutions. This bound is composed of two terms. The first one is related to $\tau$, the noise level, and decreases with the regularization parameter as $1/\lambda$. The second one is related to $\beta$ in the source condition, and increases with the regularization parameter as $\lambda^{2\alpha}$. The choice of the parameter $\lambda_*$ is then obtained by minimizing this upper bound in $\lambda$. Once we plug $\lambda_*$ in \eqref{apriori_spectral}, we obtain the bound in \eqref{final_bound}.
These results are analogous to the ones usually obtained in the deterministic setting (see for instance Corollary 4.4 in \cite{engl}), and are known to be optimal in the sense of Definition 3.17 in \cite{engl}. 

Next, we show that the regularization parameter on the grid learned from data, namely $\widehat{\lambda}_{\Lambda}$ defined in \eqref{eq:lambdahat}, achieves a similar performance to the one of $\lambda_*$. Indeed, with the aid of the previous results, and in combination with Theorem~\ref{genthm}, we obtain a sharp error bound for the regularized solution with $\widehat{\lambda}_{\Lambda}$. Toward this end, we consider the truncation operator $T:\cX\to \cX$ defined for all $x\in \cX$ as
\be\label{truncation}
Tx= 
\begin{cases}
	\quad \displaystyle{x},  & \|x\|_\cX\le1, \\
	\displaystyle\frac{x}{\|x\|_\cX}, & \|x\|_\cX>1.
\end{cases}
\ee
To apply the result in Section \ref{sec:learn1param}, we consider the loss function defined, for every $(x, x') \in \cX^2$, as
\be\label{sq_loss}
\ell(x,x')= \|Tx-Tx'\|^2_\cX.
\ee
Then, the corresponding expected risk is, for every measurable function $f$, 
\be\label{eq:eespec}
L(f)=\mathbb{E}[\|Tf(Y)-TX\|_\cX^2].
\ee
Under Assumption~\ref{aroughsnr}, for every $\lambda\in(0, +\infty)$ let $f_\la(Y)= X_\la$, where $X_\la$ is defined as in \eqref{eq:specfunctions}. Now, we next study the error obtained in this context by choosing $\la$ with ERM. 

Consider a finite set of independent and identical copies $(Y_i, X_i)$, $i=1,..., n$, of the pair $(Y, X)$ distributed as in \eqref{eq:linearip} and let $X_\la^i:=f_\la(Y_i)$. Then, the corresponding ERM is given by 
\begin{equation}\label{eq:ermspec}
\widehat{\lambda}_\Lambda\in\argmin_{\la\in \Lambda}\frac 1 n \sum_{i=1}^n \|T X_\la^i-X_i\|^2_\cX,
\end{equation}
where we used that $X_i=TX_i$ a.s.. since $\|X\|_\cX\leq 1$ almost surely. 

The following corollary provides the desired error estimates.

\begin{cor}\label{cor:spec} Let Assumption~\ref{aroughsnr} be satisfied with $\lambda_*$ as in Theorem~\ref{thm:apriori_spectral}. Suppose that Assumptions~\ref{lip_noise},~\ref{lip_source} and~\ref{lip_spectral_reg} hold, and choose the loss as in~\eqref{sq_loss}. Let $\eta\in(0,1)$. Then, with probability at least $1-\eta$,
$$
L(X_{\wh{\la}_\Lambda})
\leq \frac{2(2\alpha+Q^{2\alpha+1})}{Q}\left[\left(\frac{C_1^2}{2\alpha}\right)^{2\alpha}C_2^2\right]^{1/(2\alpha+1)} \left(\frac{\tau^{2\alpha}}{\beta}\right)^{2/(2\alpha+1)} + \frac{26}{n} \log\frac{2 N}{\eta}.
$$	
\end{cor}
 In this setting, Assumption \ref{ass:ell} is trivially satisfied. The proof will therefore consist in verifying that also Assumption~\ref{abound} holds, so that Theorem~\ref{genthm} can be applied.
\begin{proof} In this case, Assumption~\ref{ass:ell} is satisfied with $M=4$. We just need to show that Assumption~\ref{abound} is satisfied for $f_\lambda=X_\lambda$ and $L$ defined as in \eqref{eq:eespec}. Since $T$ is a projection, it is 1-Lipschitz. Then, for all measurable functions $f:\cY\to\cX$, 
$$
L(f)=\E[\|Tf(Y)-TX\|^2_\cX]\le\E[\|f(Y)-X\|^2_\cX].
$$
Then, if we define $U(\lambda)$ as the right hand side of equation~\eqref{apriori_spectral}, \eqref{bound} holds. In addition, $\la_*$ defined as in Theorem~\ref{thm:apriori_spectral} is the minimizer of $U$. Now, define the function
$$
C:[1, +\infty)\to [0,+\infty);\quad C(q):=\frac{2\alpha+q^{2\alpha+1}}{q(2\alpha+1)},
$$
and observe that it is non decreasing. Then, from the error bound~\eqref{final_bound}, we derive, for any $q\in [1, +\infty)$, that
$$
\begin{aligned}
U(q\lambda_*)&= C(q)U(\lambda_*)=\frac{2\alpha+q^{2\alpha+1}}{q}\left[\left(\frac{C_1^2}{2\alpha}\right)^{2\alpha}C_2^2\right]^{1/(2\alpha+1)} \left(\frac{\tau^{2\alpha}}{\beta}\right)^{2/(2\alpha+1)}.
\end{aligned}
$$
Hence, Assumption \ref{abound} is satisfied. The result follows by applying Theorem~\ref{genthm}. 
\end{proof}

Corollary~\ref{cor:spec} shows that, under a natural generalization of the classical assumptions in deterministic inverse problems to the stochastic setting, the error obtained with the optimal parameter on the grid for the empirical risk, namely $\wh \la_{\Lambda}$, is close to that of $\la_*$, up to a logarithmic factor that increases very slowly with $N$, and decreases with $n$. 
We add one final remark for this section.

\begin{remark}[Comparison with Theorem 4.1 in \cite{ratti2021}] The paper \cite{ratti2021} aims to learn the optimal Tikhonov regularizer, of the form $\|B(\cdot-h)\|^2$, for a linear operator $B$ and a bias vector $h\in\cX$. The main result of \cite{ratti2021} is Theorem 4.1, which establishes an excess risk bound for parameters $(\hat{B},\hat{h})$ learned by minimizing the empirical risk. The setting is quite different since, in \cite{ratti2021}, the authors learn a general Tikhonov regularizer by demonstrating that the optimal pair $(B^*, h^*)$ consists of the covariance operator and the mean of $X$, respectively. In this paper, we only learn the regularization parameter, but our setting allows for a large class of spectral filters. The assumptions of theorem 4.1, as seen in (20) and (21) of \cite{ratti2021}, are quite different from Assumption \ref{lip_source} and Assumption \ref{lip_spectral_reg}, making a direct comparisong between our Corollary $1$ and Theorem 4.1 not meaningful. We only observe that the proof of Theorem 4.1 in \cite{ratti2021} relies on learning techniques that exploit the Lipschitz continuity of the Empirical Risk with respect to the pair $(h, B)$ and a classic covering argument. In this paper, we use instead a different approach introduced in \cite{capon} for the cross-validation method.
\end{remark}

\section{Tikhonov regularization for non linear inverse problems}\label{sec:nonlin}
Next, we consider the problem of selecting the regularization parameter for Tikhonov regularization in the setting of nonlinear inverse problems \cite{engl}. 
Let $\cX, \ \cY$ be real and separable Hilbert spaces, and $A:\dom(A)\subseteq\cX\to \cY$ be a (nonlinear) operator whose domain has nonempty interior. 
Let $X, \ \eps$ be a pair of random variables with values in $\cX$ and $\cY$ respectively, and let 
\begin{equation}\label{eq:nonlinearip}
Y= A(X)+\eps, \quad \text{ a.s.} 
\end{equation}
with $X\in\mathrm{int}(\dom(A))$ almost surely. We make several assumptions. The first one is on the noise $\eps$. 
\begin{ass}\label{ass:bnoise} There exists a constant $\tau >0$ such that
\[
\E[\|\eps\|^2_\cY|X]\le \tau^2 \quad\mathrm{ a.s.} 
\]
\end{ass}
Using Jensen's inequality for the conditional expectation \cite[9.7 (h)]{williams}, we derive from the previous assumption that
\be\label{noise_bound1}
\E[\|\eps\|_\cY|X]\le \tau \text{ a.s.} 
\ee
Next we impose fairly standard conditions on the operator $A$. 
\begin{ass}\label{nlip_lipschitz}
The operator $A: \dom(A)\to \mathcal{Y}$ is a continuous and weakly closed operator with $\mathrm{int}(\dom(A))$ non-empty, and with $\dom(A)$ convex. Moreover, $A$ is Fr\'echet differentiable in $\mathrm{int}(\dom(A))$ with derivative denoted by $A'$ and  there exists a constant $C_0>0$ such that, for all $x$ and $x'\in \mathrm{int}(\dom(A))$, 
\be\label{smooth}
\|A'(x)-A'(x')\|_{\mathrm{op}} 
\le C_0\|x-x'\|_\cX.
\ee
\end{ass}
The previous assumption implies that, for all $x\in \mathrm{int}(\dom(A))$ and $x'\in\dom(A)$,
$$
\|A(x')- A(x)- A'(x)(x'-x)\|_\cY \le \frac{C_0}{2} \|x'-x\|^2_\cX,
$$
so that, by the triangle inequality,
\be\label{useful}
\|A'(x)(x'-x)\|_\cY\le \|A(x')- A(x)\|_\cY + \frac{C_0}{2} \|x'-x\|^2_\cX.
\ee
Here, we assume global Lipschitz continuity of the derivative to avoid technicalities, but the argument could be extended under a local smoothness assumption as in \cite{clasonreg}. 

For nonlinear inverse problems, the Tikhonov estimator is defined with respect to a suitable initialization.
Here, we assume the initialization to be described by a random variable $X_0$ with values in $\cX$. In addition, the set of minimizers 
$$
\argmin_{x\in \dom(A)}  \|A(x)- Y(\omega)\|_\cY^2+\la\|x-X_0(\omega)\|^2_\cX
$$ 
is nonempty for every $\omega\in\Omega$ thanks to  Assumption~\ref{nlip_lipschitz}, see \cite[Theorem 10.1]{clasonreg}. A corresponding Tikhonov regularized estimator is a random variable $X_\la$ defined by setting, for almost all $\omega\in\Omega$ 
\be\label{eq:minnltik}
X_\la(\omega) \in \argmin_{x\in \dom(A)}  \|A(x)- Y(\omega)\|_\cY^2+\la\|x-X_0(\omega)\|^2_\cX.
\ee
Note that $X_\la$ depends on $Y$ and $X_0$, but we will omit this dependence for the sake of simplicity.    
The existence of a random variable $X_\la$ taking values in the set of minimizers is ensured under some additional assumptions, see e.g. Filippov's Implicit function Theorem \cite[Theorem 7.1]{Himmelberg}. For that reason, we directly  assume that such measurable selection exists. \\
The following assumption will be needed to derive the error bounds and extends analogous conditions in the deterministic case.
\begin{ass}\label{nlip_source}
	The random variable $X$ is such that $\|X-X_0\|_{\cX}\leq 1$ and, under Assumption~\ref{nlip_lipschitz}, there exists a random variable $Z$ with values in $\cY$, $\beta>0$ such that almost surely
	$$
	X-X_0= A'(X)^* Z,
	$$
	and 
	$$
	\|Z\|_\cY\le \beta \text{ a.s.}, \quad \text{with}\quad \beta C_0 < 1,
	$$
 where $C_0$ is the constant introduced in Assumption~\ref{nlip_lipschitz}.
\end{ass}
The latter assumption can be seen as a nonlinear version of the source condition considered in Assumption~\ref{lip_source} (for $s=1$).

In the next result, which is  analogous to Theorem \ref{thm:apriori_spectral},  we derive a bound on the error of the Tikhonov regularized solution, leading to a priori parameter choices. 
\begin{thm}\label{thm:apriori_nonlin}
	Suppose that Assumptions~\ref{ass:bnoise},~\ref{nlip_lipschitz} and~\ref{nlip_source} are satisfied. Then  the following bound holds: for all $\la\in(0, +\infty)$, 
	\be\label{apriori_nonlin}
	\E[\|X_\la- X\|_\cX^2]\le\frac{(\tau+\beta\la)^2}{(1-\beta C_0)\la}.
	\ee
	In particular, setting $\lambda_*=\tau/\beta$,
	$$
	\E[\|X_{\la_*}- X\|_\cX^2]\le 4(1-\beta C_0)^{-1}\tau\beta.
	$$
\end{thm}
The proof is a modification of the one in the deterministic setting, see e.g. \cite{clasonreg,engl}. 
\begin{proof} 
The expressions below are all intended to hold almost surely.
By definition of $X_\la$, $X$ and $\eps$, it follows that
\begin{align}\label{eq:argbas}
\nonumber\|A(X_\la)- Y\|_\cY^2+\la\|X_\la-X_0\|^2_\cX&\le\|A(X)- Y\|_\cY^2+\la\|X-X_0\|^2_\cX\\
&=\|\eps\|_\cY^2+\la\|X-X_0\|^2_\cX.
\end{align}
Since
\begin{equation}\label{eq:addsub}
\|X_\la-X_0\|^2_\cX= \|X_\la- X\|^2_\cX+\|X-X_0\|^2_\cX+2\langle X_\la- X, X-X_0\rangle_\cX,    
\end{equation}
inequality \eqref{eq:argbas} implies
$$
\|A(X_\la)- Y\|_\cY^2+\la\|X_\la-X\|^2_\cX \le\|\eps\|_\cY^2- 2\la\langle X_\la- X, X-X_0\rangle_\cX.
$$
Then, Assumption~\ref{nlip_source} and Cauchy-Schwartz inequality yield
\be\label{error_dec}
\|A(X_\la)- Y\|_\cY^2+\la\|X_\la-X\|^2_\cX \le\|\eps\|_\cY^2+2\la\|A'(X)(X_\la- X)\|_\cY\|Z\|_\cY.
\ee
Since  $X\in\mathrm{int}(\dom(A))$ and $X_\la\in\dom(A)$, and $\dom(A)$ is convex by assumption, inequality \eqref{useful} with $x=X$ and $x'=X_\la$ yields 
$$
\|A'(X)(X_\la- X)\|_\cY\le \|A(X_\la)- A(X)\|_\cY
+\frac{C_0}{2}\|X_\la-X\|^2_\cX,
$$
so that, by adding and subtracting $Y$ in the first term of the right hand side, we obtain
$$
\|A'(X)(X_\la- X)\|_\cY\le \|A(X_\la)- Y\|_\cY+\|\eps\|_\cY
+\frac{C_0}{2}\|X_\la-X\|^2_\cX.
$$
Plugging the above inequality into~\eqref{error_dec}, we get 
\begin{align*}
\|A(X_\la)- Y\|_\cY^2+\la\|X_\la-X\|^2_\cX &\le\|\eps\|_\cY^2+ 
2\la\|Z\|_\cY (\|A(X_\la)- Y\|_\cY\\
&\quad+\|\eps\|_\cY
+\frac{C_0}{2}\|X_\la-X\|^2_\cX).
\end{align*}
By adding $\la^2\|Z\|_\cY^2$ to both sides and rearranging the terms, we get
\begin{align*}
\left(\|A(X_\la)- Y\|_\cY-\la\|Z\|_\cY\right)^2+\la\|X_\la-X\|_\cX^2&\leq \|\eps\|_\cY^2+2\la\|Z\|_\cY\large(\|\eps\|_\cY\\
&\quad+\frac{C_0}{2}\|X_\la-X\|_\cX^2)+\la^2\|Z\|_\cY^2.
\end{align*}
Next, we take expectations on both sides. First, recall that Assumption~\ref{ass:bnoise} implies \eqref{noise_bound1}, i.e. $\E[\|\eps\|_\cY]\leq \tau$ and therefore, with Assumption~\ref{nlip_source},
$$
\E[\|Z\|_\cY\|\eps\|_\cY]\leq \beta\tau.
$$
Assumption~\ref{nlip_source} implies also that
$$
\begin{aligned}
\E[\|Z\|_\cY\|X_\la-X\|_\cX^2]&\leq\beta\E[\|X_\la-X\|_\cX^2].
\end{aligned}
$$
We then get that
\begin{align*}
\E[\left(\|A(X_\la)- Y\|_\cY-\la\|Z\|_\cY\right)^2]+\la\E[\|X_\la-X\|_\cX^2]&\le \tau^2+2\la\beta\tau\\
&\quad+\la^2\beta^2+\la C_0\beta\E[\|X_\la-X\|_\cX^2].
\end{align*}
In particular, 
$$
\E[\|X_\la-X\|^2_\cX ]\le (1-\beta C_0)^{-1} \frac{(\tau+\beta\la)^2}{\la},
$$
where we used the assumption that $\beta C_0<1$. Finally, the value of $\la$ that minimizes the above bound is
$$
\la_*=\frac{\tau}{\beta},
$$
and the corresponding error bound is
$$
\E[\|X_{\la_*}- X\|_\cX^2]\le 4(1-\beta C_0)^{-1}\tau\beta,
$$
which proves the result.
\end{proof}
To apply Theorem~\ref{genthm}, we consider the problem obtained with a truncated square loss:
\begin{equation} \label{nonlintrunc}
\ell(x,x')=\|T(x-X_0)-T(x'-X_0)\|_\cX^2,    
\end{equation}
where $T$ is the truncation operator defined in~\eqref{truncation}.
The corresponding expected risk is given by
\[
L(f)=\mathbb{E}[\|T(f(Y)-X_0)- T(X-X_0)\|_{\cX}^2].
\]
We focus on Tikhonov regularization, where, for every $\la\in(0, +\infty)$, $f_\la(Y)=X_\la$ is given by \eqref{eq:minnltik}, and analyze the error corresponding to the choice of the regularization parameter with ERM. Consider independent and identical copies $(Y_i, X_i)$, $i=1,..., n$, of the pair of random variables $(Y, X)$ as in \eqref{eq:nonlinearip}. The ERM problem is given by  
 \begin{equation}
\widehat{\lambda}_\Lambda\in\argmin_{\la\in \Lambda}\frac 1 n \sum_{i=1}^n \|T (X_\la^i-X_0)-(X_i-X_0)\|^2_\cX.
 \end{equation}
where $X_\la^i:=f_\la(Y_i)$. In the following result we derive an upper bound corresponding to the expected risk. 
\begin{cor}\label{cor:nonlin} 
    Suppose that  Assumptions~\ref{ass:bnoise},~\ref{nlip_lipschitz} and~\ref{nlip_source} hold. Let Assumption~\ref{aroughsnr} be satisfied with $\lambda_*=\tau/\beta$ and choose the loss as in \eqref{nonlintrunc}. Let $\eta\in(0,1)$. Then, with probability at least $1- \eta$,
	$$
	L(X_{\widehat{\lambda}_\Lambda})\le \frac{(1+Q)^2}{2Q(1-\beta C_0)}\tau\beta + \frac{26}{n} \log\frac{2 N}{\eta}.
	$$
\end{cor}

\begin{proof} To prove the result, it is enough to show that Assumptions \ref{ass:ell} and \ref{abound} are satisfied. First, note that Assumption~\ref{ass:ell}  is satisfied since the truncated square loss in \eqref{nonlintrunc} is bounded by $4$. Moreover, since $T$ defined in~\eqref{truncation} is the projection on a convex and closed set, it is $1-$Lipschitz, so that Theorem~\ref{thm:apriori_nonlin} implies
	$$
	L(X_\la)\le\E[{\|X_\la-X\|^2_\cX}]\leq U(\lambda),
	$$	
	with $U(\lambda)=(1-\beta C_0)^{-1} (\tau+\beta\la)^2\la^{-1}$. The minimizer of $U$ is $\lambda_*=\tau/\beta$ with $U(\lambda^*)=4(1-\beta C_0)^{-1}\tau\beta$ and, for every $q\ge 1$ we have that
    $$U(q\la_*)=\frac{(1+q)^2}{q}(1-\beta C_0)^{-1}\tau\beta = \frac{(1+q)^2}{4q}U(\lambda^*).$$
    Since the function $$C:[1, +\infty)\to[0,+\infty);\quad C(q):=\frac{(1+q)^2}{4q}$$ is non decreasing, Assumption~\ref{abound} is satisfied. The result then follows from Theorem~\ref{genthm}.
\end{proof}

Corollary~\ref{cor:nonlin} establishes an upper bound on the expected risk of $X_{\wh{\lambda}_{\Lambda}}$, corresponding to the choice of the optimal regularization parameter based on ERM in the grid $\Lambda$. Actually, it ensures that the error obtained when considering $\wh{\lambda}_{\Lambda}$ is close to that of $\lambda_*$, except for an additive error term that decreases with $n$. Notably, the dependence on the cardinality of the grid $N$ is only logarithmic.

\section{Variational regularization with convex regularizers for linear inverse problems}\label{sec:convex}

In this section, we consider the linear inverse problem setting in Section~\ref{sec:spectral}, with Assumption~\ref{lip_noise} on the noise. We study variational regularization with a general function $J:\cX\to \R\cup\{+\infty\}$  instead of the squared norm, 
\be\label{tik_conv}
X_\la(\omega)\in \argmin_{x\in \cX} \frac{1}{2} \|Ax- Y(\omega)\|_\cY^2+\lambda J(x).
\ee
In this section, we assume that the set of minimizers of the function 
$$
x\mapsto  \|Ax- Y(\omega)\|_\cY^2/2+\lambda J(x)
$$ 
is nonempty for almost every $\omega\in\Omega$, and  that  $\omega\mapsto X_\la(\omega)$ is a measurable selection of the set of minimizers.
This setting includes various examples of sparsity-inducing regularizers beyond Hilbertian norms, see e.g. \cite{burgerbenning} for references. We discuss specific examples in Sections \ref{sec:l1} and \ref{sec:legendre}. For this class of regularization schemes, a natural error metric is given by the Bregman divergence, defined for every $x$, $x'\in\cX$ as
\begin{equation}\label{eq:bregdist}
D_J(x,x')=
\begin{cases}
	J(x)-J(x')-\langle s_J(x'), x-x'\rangle_\cX, & \text{ if } x'\in\mathrm{int}(\dom J),\\
	\quad +\infty, & \text{ elsewhere},
\end{cases}
\end{equation}
where $s_J(x')$ is an element of $\partial J(x')$, which is nonempty as long as $x'\in\mathrm{int}(\dom J)$ \cite[Theorem 9.23]{bauschcomb}. If $x$ and 
$x'$ belong to $\mathrm{int}(\dom J)$, we can consider also the symmetric Bregman distance, that is
$$
d_J(x, x')= D_J(x,x')+D_J(x',x)=\langle s_J(x)- s_J(x'), x-x'\rangle_\cX.
$$ 
Of course, if $J$ is not differentiable, both the Bregman divergence and the symmetric one depend on the choice of the specific subgradient $s_J(x)$ (and $s_J(x')$). To derive an error bound we consider the following assumptions.
\begin{ass}\label{ass:J}
	The function $J:\cX\to \R$ is proper, convex, lower semicontinuous and satisfies $\dom (\partial J)=\mathrm{int}(\dom (J))$.
\end{ass} 
The previous assumption is satisfied, in particular, in two  settings, which are discussed in the following: the one where $\dom J=\mathbb{R}^d$ and the one where $J$ is essentially smooth.
\begin{ass}\label{lipc_source}
	The random variable $X$ takes values in $\mathrm{int}(\dom (J))$ a.s.. Moreover, we assume that there exists a random variable $Z\in\cY$, is measurable with respect to the $\sigma$-algebra generated by $X$, such that $A^*Z\in\partial J(X)$ a.s.. Finally, we assume that there exists $\beta>0$ such that
	$$
	\E[\|Z\|^2_\cY]\le \beta^2.
	$$
\end{ass}
Assumption~\ref{lipc_source} can be seen as a generalization of the source condition for the squared norm regularization in Assumption~\ref{lip_source}, in the case $s=1$. In the following, we will analyze the behavior of $d_J(X_\la,X)$. We first show that this quantity is well-defined. 
From the optimality condition for the Tikhonov problem~\eqref{tik_conv} we derive that, almost surely,
\begin{equation}\label{inclusion}    
	 \frac{1}{\la} A^*(Y-AX_\la) \in  \partial J(X_\la).
\end{equation}
In particular we know that $X_\la\in \dom \partial J$ and so, by Assumption~\ref{ass:J}, that $X_\la\in\mathrm{int}(\dom J)$. Moreover, from Assumption \ref{lipc_source} we have that $X\in \mathrm{int}(\dom J)$ almost surely, and
 $$A^*Z\in\partial J(X).$$
 Then, the symmetric Bregman distance is well defined, and can be written as
 \begin{align}
 \label{eq:Breg}
        d_J(X_\la,X)&=\langle \frac{1}{\la}A^*(Y-AX_\la)-A^*Z, X_\la- X\rangle_\cX.
\end{align}
The Bregman divergences we consider (both the symmetric and the standard one) are based on the specific subdifferentials considered in the latter formula. In the setting above, we have the following upper bound. 

\begin{thm}\label{thm:apriori_convex}
Under Assumptions~\ref{lip_noise}, \ref{ass:J} and~\ref{lipc_source} the following bound holds, for all $\la\in(0, +\infty)$,
\begin{equation}\label{apriori_lipc}
\E[d_J(X_{\la}, X)]\le  \frac{\tau^2}{2\la}+\frac{\beta^2\la}{2}.
\end{equation}
In particular, taking  $\la_*=\tau/\beta$, we have
\begin{equation}\label{apriori_opt_lipc}
\E[d_J(X_{\la_*},X)]\le\beta \tau.
\end{equation}
\end{thm}
\begin{proof}
The identities and inequalities below are intended to hold almost surely.
By Assumption \ref{lipc_source},
	\begin{align*}
		\la d_J(X_\la,X)+\|A(X_\la- X)\|^2_\cY&=\langle A^*(Y-AX_\la)-\la A^*Z, X_\la- X\rangle_\cX\\
        &\quad+\|A(X_\la- X)\|^2_\cY\\ 
		&=\langle Y-AX_\la-\la Z+AX_\la-AX, A(X_\la-X)\rangle_\cY\\
        &=\langle Y-AX-\la Z, A(X_\la-X)\rangle_\cY\\
        &\leq \frac 1 2 \|Y-A X-\la Z\|^2_\cY +\frac 1 2 \|A(X_\la-X)\|_\cY^2.
	\end{align*}
	Rearranging the terms, we obtain
	$$
	\la d_J(X_\la,X)+\frac 1 2 \|A(X_\la- X)\|^2_\cY\le\frac 1 2 \|Y-A X-\la Z\|^2_\cY.
	$$
	Taking the conditional expectation with respect to $X$, we get
	\begin{align*}
		\la\E[d_J(X_\la,X)|X]+\frac 1 2 \E [\|A(X_\la- X)\|^2_\cY|X]&\le\frac 1 2 \E[\|Y-AX\|^2_\cY|X]+\frac{\la^2}{2}\E[\|Z\|^2_\cY|X]\\
		&\quad-\lambda \E[\langle Y-A X, Z\rangle_\cY|X].
	\end{align*}
	By Assumption~\ref{lipc_source}, $Z$ is a measurable function with respect to $X$, and therefore the last term is zero since $Y=AX + \varepsilon$ and by Assumption~\ref{lip_noise}. Thus, if we take the full expectation, the previous inequality implies
 \[
	\begin{aligned}
	\la\E[d_J(X_\la,X)]+\frac 1 2 \E [\|A(X_\la- X)\|^2_\cY]&\le\frac 1 2 \E[\|Y-AX\|^2_\cY]+\frac{\lambda^2}{2}\E[\|Z\|^2_\cY]\\
	&\leq\frac{\tau^2 }{2}+\frac{\beta^2\la^2}{2},
	\end{aligned}
 \]
	by Assumptions~\ref{lip_noise} and~\ref{lipc_source}. Therefore,
 \begin{equation}
     \label{eq:bbreg}
     \E[d_J(X_\la,X)]\leq \frac{\tau^2}{2\la}+\frac{\beta^2\la}{2}.
 \end{equation}
 The value of $\la$ minimizing the above upper bound is
	$$
	\la_*=\frac{\tau}{\beta}.
	$$ and the theorem follows.
 \\
\end{proof}

\begin{rem}
Following \cite{rebuhe}, the above analysis can be extended considering $\cX$ to be a Banach space embedded in a Hilbert space. In this case, the inner product in $\cX$ needs to be replaced by the corresponding duality pairing. 
\end{rem}

In the rest of the section, we will apply Theorem~\ref{genthm} to different loss functions, all based on the Bregman divergence.
To perform the analysis, additional assumptions are needed on $J$ to ensure that the hypotheses of Theorem~\ref{genthm} are satisfied, e.g. the boundedness of the loss. 
We focus on two different settings: the case of sparsity inducing regularizers, of the form $J(x)=|Gx|$, where $G$ is a general linear and bounded operator and $|\cdot|$ a general norm (for instance, the $\ell^1$-norm), and the case of regularizers $J$ of Legendre type.

\subsection{Sparsity inducing regularizers}\label{sec:l1}
In this section, we focus on the finite-dimensional setting, where $\cX = \mathbb{R}^d$, $1\leq d< +\infty$. We study sparsity-inducing regularizers such as the $\ell 1$ norm \cite{bach2012}. Towards this end, we first introduce a generic norm on $\mathbb{R}^m$ (not necessarily the euclidean one), which we denote by $|\cdot|$, and the corresponding dual norm $|\cdot|_*$. We then fix a linear and bounded operator $G\colon(\cX,\|\cdot\|)\to (\R^m, |\cdot|)$.
We will consider the following structural assumption.
\begin{ass}\label{ass:norm} The regularizer $J\colon \mathbb{R}^d\to\mathbb{R}$ is defined by setting, for every $x\in\mathbb{R}^d$,
\begin{equation}
J(x)=|Gx|,
\end{equation}
and  $\|G\|_{\mathrm{op}}\leq R$, for some $R>0$ (here the operator norm is meant with respect to the spaces $\cX = \mathbb{R}^d$ and $\mathbb{R}^m$ with their norms $\|\cdot\|$ and $|\cdot|$, respectively).  
\end{ass}

The above condition describes the class of sparsity inducing regularizers we consider, including Lasso \cite{tibshi96} ($G$ equal to the identity and $|\cdot|$ the $\ell^1$ norm),  Graph-Lasso \cite{meinbuhl06}, penalties for multitask learning \cite{MosRosSan10}, group lasso \cite{SalVil21}, $\ell q$ penalties \cite{grassmair2008}, and Total Variation regularization \cite{rof92}, among others (see \cite{hastibwain15} and references therein). For these regularization functions $J$, the subdifferential can be written as
$$
\partial J(\cdot) = G^*\partial |\cdot|(G\cdot),
$$
which is nonempty at every point $x \in \cX$. In addition, recall that the subdifferential of the norm can be computed as \cite[Remark 1.1]{bach2012}
\[
\partial |\cdot|(x)=\{\eta\in\mathbb{R}^m\,:\quad \langle\eta,x\rangle=|x|, \ |\eta|_*\leq 1\}.
\]
In this section, we consider the loss function defined by the Bregman divergence for every $x$ and $x'\in\R^d$:
\begin{equation}\label{eq:bregmanloss}
\ell(x, x')=D_J(x, x')
\end{equation}
where $D_J$ is defined as in \eqref{eq:bregdist}, for some subgradient $s_J(x')\in\partial J(x')$. 
As before, if we let $f_\la(Y)= X_\la$, then the corresponding expected risk is given by
\begin{equation}\label{eq:ersparse}
L(X_\la)=\mathbb{E} [D_J(X,X_\la)].
\end{equation}
In this case, and as in Section \ref{sec:spectral}, we also assume that the random variable $X$ is such that $\|X\|\leq 1$ a.s.. Finally, if we denote $X_\la^i:=f_\la(Y_i)$, then the ERM is given by 
\begin{equation}\label{eq:ermsparse}
\wh{\lambda}_\Lambda\in\argmin_{\la\in \Lambda}\frac 1 n \sum_{i=1}^n D_J(X_i,X^i_\la).
\end{equation}
The latter approach has been already applied in practice. In particular, in the case of $R$ being the Total Variation regularizer \cite{Chenchene2023}. We can now state the probabilistic error estimates for this setting.   
\begin{cor}\label{cor:sparse} In the setting of this subsection, 
	let Assumptions~\ref{lip_noise},~\ref{lipc_source} and~\ref{ass:norm} be satisfied, let Assumption~\ref{aroughsnr} be satisfied with $\la_*=\tau/\beta$ as in Theorem \ref{thm:apriori_convex} and choose the loss as in \eqref{eq:bregmanloss}. Let $\eta\in(0,1)$. Then, with probability at least $1-\eta$,
\begin{equation}\label{eq:LB}
L(X_{\wh{\la}_{\Lambda}})\leq \frac{1+Q^{2}}{Q} \beta\tau+ \frac{13 R}{n} \log\frac{2 N}{\eta}.
\end{equation}
\end{cor}
\begin{proof}  To apply Theorem~\ref{genthm}, we need to check that Assumptions~\ref{ass:ell} and \ref{abound} are satisfied. For every $x\in\R^d$ with $\|x\| \leq 1$ and $z\in\R^d$, we have 
\begin{align*}
D_J(x, x')&=|Gx|-|Gx'|-\langle G^*s_{|\cdot|}(Gx'), x-x'\rangle_{\R^m}\\
&=|Gx|-|Gx'|-\langle s_{|\cdot|}(Gx'), Gx-Gx'\rangle_{\R^m}\\
&=|Gx|-\langle s_{|\cdot|}(Gx'), Gx\rangle_{\R^m}\\
&\leq (1+|s_{|\cdot|}(Gx')|_*)|Gx|\\
& \leq 2\|G\|_{\mathrm{op}}\|x\|\\
&\leq  2 R.
\end{align*}
Hence, the loss function is bounded on the cylinder $\{(x,x')\in\mathbb{R}^{d\times d}\,:\, \|x\|\leq 1\}$ , and Assumption \ref{ass:ell} is therein satisfied with $M=2 R $. 
We are left to show that Assumption~\ref{abound} is satisfied for $f_\lambda(Y)=X_\lambda$ and $L$ defined as in \eqref{eq:LB}.
From the inequality
$$
D_J(X, X_\la)\leq d_J(X, X_\la)
$$
and Theorem \ref{thm:apriori_convex}, we derive that
\[L(X_{\la}) \leq U(\la),\]
where $U(\la)=\tau^2/(2\la)+\beta^2\la/2$. The latter is  minimized by $\lambda_*=\tau/\beta$ and satisfies
$$
U(q\lambda_*)\le \frac{1+q^2}{2q}\beta\tau,
$$
where the multiplicative factor depending on $q$ is a nondecreasing function for $q\geq 1$. The statement then follows from Theorem~\ref{genthm}. 
\end{proof}

\subsection{Legendre Regularizers}\label{sec:legendre}
In this section, we consider Legendre regularizers. We start by recalling some definitions, see \cite{bauschke2001} for more details. A proper, convex and lower semicontinuous function $J\colon\cX\to\mathbb{R}\cup\{+\infty\}$ is said to be essentially smooth if $\partial J$ is locally bounded and single valued on its domain. The function $J$ is essentially strictly convex if $(\partial J)^{-1}$ is locally bounded on its domain and $J$ is strictly convex on every convex subset of $\dom \partial J$. A function $J$ is Legendre if it is proper, lower semicontinuous and it is both essentially smooth and essentially strictly convex. 
In this section, we will rely on the following assumption. 

\begin{ass}\label{ass:Leg} The function $J\colon \cX\to \mathbb{R}\cup\{+\infty\}$ is Legendre.
\end{ass}
In particular, Assumption \ref{ass:Leg} implies Assumption \ref{ass:J} by \cite[Theorem 5.6]{bauschke2001}. 
Now, consider $x_0\in \mathrm{int}(\dom J)$ and $r>0$ such that the ball centered at $x_0$ with radius $r$, $B:=\{x\in \cX\,:\, \|x-x_0\|\leq r\}$, is a subset of $\mathrm{int}(\dom J)$. Again, by Assumption \ref{ass:Leg}, it is possible to define the projection onto $B$ with respect to the Bregman divergence for every $x\in\mathcal{X}$ (see \cite[Corollary 7.9]{bauschke2001}), by setting 
\begin{equation}\label{eq:bregproj}
\pi_B(x):=\argmin_{z\in B} D_J(z, x).    
\end{equation}
In this setting, the Bregman projection is univocally defined, meaning that it does not depend on the choice of the subgradient. Indeed, if $x \notin \mathrm{int}(\dom J)$, then $D_J(z,x)=+\infty$. Otherwise, $x \in \mathrm{int}(\dom J)=\dom(\partial J)$, where the subdifferential of $J$ is single valued. Moreover, by definition, $\pi_B(x) \in B \subseteq\mathrm{int}(\dom J)$. Recalling that it always holds $\mathrm{int}(\dom J)\subseteq \dom (\partial J)$, we know that the subdifferential of $J$ is non empty at each point of $B$. In particular, under Assumption \ref{ass:Leg}, the subdifferential of $J$ is single valued on $B$. Then, for every $x\in B$ we denote by $\nabla J(x)$ the subdifferential of $J$ at $x\in B$.
We need an additional assumption on the function $J$ on the set $B$, namely a uniform upper bound for the norm of its gradient; i.e. of $\nabla J$.
\begin{ass}\label{ass:gradbd} There exists $R>0$ such that
\[
\sup_{x\in B}\|\nabla J(x)\|\leq R.
\]
\end{ass}
Note that, since $J$ is Legendre and essentially smooth, then $\nabla J$ is locally bounded  on $\mathrm{int}(\dom J)$. This means that for every $x\in \mathrm{int}(\dom J)$ there exists $\varepsilon>0$ such that $\sup_{z\in B_{\varepsilon}(x)} \|\nabla J(z)\| <+\infty$, but this does not imply the validity of Assumption~\ref{ass:gradbd}. In this context, we consider the loss function defined for all $x,x' \in \cX$ as the Bregman divergence between the projections onto $B$, namely
\be\label{bregman_loss}
\ell(x,x')= D_J(\pi_B(x),\pi_B(x')), 
\ee
which is univocally defined since $\pi_B(x')\in B$, and the subdifferential of $J$ is non empty and single valued on $B$. 
We consider also the corresponding expected risk, defined as
\[
L(f) =\E[D_J(\pi_B(X),\pi_B(f(Y)))].
\]
In this case, and in opposition with the other sections where we assumed that $\|X\|\leq 1$, we assume that $X$ is such that $X\in B$ a.s.. As in the previous sections, we want to bound the expected risk of the regularization method $f_\la(Y)= X_\la$ defined as in \eqref{tik_conv}, when $\lambda$ is selected by ERM,
$$
\wh{\lambda}_\Lambda\in\argmin_{\la\in \Lambda}\frac 1 n \sum_{i=1}^n D_J(\pi_B(X_i),\pi_B(X_\la^i))),
$$ 
where $X_\la^i:=f_\la(Y_i)$. The corresponding error bound is given in the following corollary.
\begin{cor}\label{cor:convex}
Let Assumptions~\ref{lip_noise},~\ref{lipc_source},~\ref{ass:Leg} and ~\ref{ass:gradbd} be satisfied, let Assumption~\ref{aroughsnr} be satisfied with $\la_*=\tau/\beta$ as in Theorem \ref{thm:apriori_convex} and choose the loss as in \eqref{bregman_loss}. Let $\eta\in(0,1)$. Then, with probability at least $1-\eta$,
$$
L(X_{\wh{\la}_\Lambda})
\leq \frac{1+Q^2}{Q}\beta\tau+ \frac{26 R r}{n} \log\frac{2 N}{\eta}.
$$
\end{cor}
\begin{proof} To prove the statement, we will rely again on Theorem~\ref{genthm}. Therefore we just need to show that Assumptions~\ref{ass:ell} and~\ref{abound} hold. 
We first show that Assumption~\ref{ass:ell} is satisfied. Since both $\pi_B(x)$ and $\pi_B(x')$ belong to $B$, and by Assumption~\ref{ass:gradbd}, recalling that $\partial J$ is single valued on $B$, it follows that
\begin{align*}
0 & \ \leq \  \ell(x,x') \  =\   D_J(\pi_B(x), \pi_B(x')) \  \leq \  D_J(\pi_B(x), \pi_B(x')) + D_J(\pi_B(x'), \pi_B(x))\\
& \  = \  \langle \nabla J(\pi_B(x))-\nabla J(\pi_B(x')),\pi_B(x) -\pi_B(x')\rangle \  \leq \   4R r.    
\end{align*}
Then, the considered loss function \eqref{bregman_loss} is bounded and Assumption~\ref{ass:ell} is satisfied with $M=4R r$. Next, we check that Assumption \ref{abound} is satisfied. First, observe that both $X$ and $X_{\la}$ belong to $\dom(\partial J)$ almost surely since $X\in B$ by assumption and  by the optimality condition stated in \eqref{inclusion}. Then, the subdifferential of $J$ is not empty (and so single valued) at $X, X_{\la}$ and
$$
d_J(X,X_\la)\ge D_J(X,X_\la)\ge D_J(X,\pi_B(X_\la))+D_J(\pi_B(X_\la),X_\la),
$$
by the first order optimality conditions of problem~\eqref{eq:bregproj} and the fact that $X\in B$. \\
Again, since $X\in B$ almost surely, we have that $\pi_B(X)=X$ almost surely. Then, the previous inequality implies that
\begin{align}\label{eq:ubo}
L(X_{\la})=\E[D_J(\pi_B(X),\pi_B(X_\la))]=\E[D_J(X,\pi_B(X_\la))]\le \E[d_J(X,X_\la)].
\end{align}
Theorem~\ref{thm:apriori_convex} gives the bound $\E[d_J(X,X_\la)] \leq U(\la)$, where $U(\la)=\tau^2/(2\la)+\beta^2\la/2$. So, together with \eqref{eq:ubo}, this implies that 
\[
L(X_\la)\leq U(\la).
\]
The minimizer of $U(\la)$ is given by $\lambda_*=\tau/\beta$ with $U(\lambda_*)=\beta\tau$.
We derive directly from the definition that
$$
U(q\lambda_*)= \frac{1+q^2}{2q}\beta\tau=\frac{1+q^2}{2q}U(\lambda_*)
$$
for any $q\geq 1$, where the multiplicative term $(1+q^2)/(2q)$ is a non decreasing function for $q\ge 1$. Hence, Assumption~\ref{abound} is satisfied and we can apply Theorem~\ref{genthm} to obtain the desired result.
\end{proof}

\section{Numerical results}\label{expers}
In this section, we provide an empirical validation of the  theoretical results discussed in the previous sections. We consider different numerocal settings and, for each of them,  analyze the behvior of the epected risk for the proposed data driven parameter choice.
First, we consider the setting of linear inverse problems with squared norm regularization. In this case, we focus on  Tikhonov regularization and the Landweber method. For both of them we compare the proposed data-driven procedure with the so-called quasi-optimality criterion \cite{baure}. Then, we turn to more general regularization penalties. More precisely, we consider the problem of denoising and deblurring sparse signals with the $\ell^1$-norm, and TV denoising for images.\\

In all experiments, the expected risk $L$ is always approximated  empirically using either $N=500$ or $N=1000$ points, depending on the complexity of the experiment. Similarly, the optimal parameter $\lambda_*$ is selected on a sufficiently fine grid to approximate the interval $(0,+\infty)$. \\

\textbf{Code details:} All of the simulations have been implemented in Python on a laptop with 32GB of RAM and 2.2 GHz Intel Core I7 CPU. In Section \ref{sec:deblurring} we also use the library Numerical Tours by G. Peyré \cite{numtours}. The code is available at \url{https://github.com/TraDE-OPT/Learning-the-Regularization-Parameter}.

\subsection{Spectral regularization methods}
\label{sec:numspec}
In this section, we empirically analyze the proposed data-driven parameter selection strategy for Tikhonov regularization and the Landweber method to solve an instance of a  linear inverse problem as in Section~\ref{sec:spectral}. We consider a problem of the form 
$$
Y=AX+\eps,
$$
which we describe next. The operator $A$ is a $70\times70$  matrix, with Gaussian entries $a_{i, j}\sim N(0, 1)$, $1\leq i$, $j\leq 70$, that will be then normalized by its operator norm, which in this case coincides with the $2$-norm. To ensure that Assumption~\ref{lip_source} is satisfied with a known exponent, we define the random variable $X\in\R^{70}$ as
$$
X=(A^*A)^sZ, 
$$
with $s>0$ to be fixed later and  $Z$ is sampled uniformly in the unit ball. This, jointly with $\|A\|_2\leq 1$, ensures that $\|X\|\leq 1$ almost surely.  Note that, in this setting, Assumption~\ref{lip_source} is satisfied with $\beta=1$. Finally, $\eps\sim N(0,\tau^2\mathrm{Id})$, which satisfies Assumption \ref{lip_noise}.
The training set is obtained  sampling $n$ independent pairs $(y_i,x_i)$ from the previous model.\\ 

The section is divided into two parts: 
\begin{itemize}
\item empirical validation of the theoretical results,
\item comparison of the studied method with the quasi-optimality criterion \cite{tikhonov1979}. \\
\end{itemize}
In both cases, every experiment is run 30 times, and we report both the mean (in solid lines) and the values between the $5^{\rm th}$-percentile and $95^{\rm th}$-percentile of the data (in shaded regions).

\subsubsection{Illustration of the data-driven parameter choice}\label{sec:verif}
We start considering the problem described in Section~\ref{sec:numspec} with noise level $\tau=10^{-2}$ and source condition $s=0.5$. Starting from a training set $\{(y_i,x_i)\}_{i=1}^{50}$, for every $\la\in\Lambda$,  we define the empirical risk  for the Tikhonov regularized solution as
\begin{equation}\label{eq:emp_err}
\widehat{L}(X_\la)=\frac{1}{50}\sum_{i=1}^{50} \|TX^i_\la-x_i\|^2,    
\end{equation}
where $X^i_\la:=(A^*A+\la I)^{-1}A^*y_i$ (see Section~\ref{sec:spectral}). The empirical risk for the Landweber method is defined analogously, where 
in this case $X^i_\la=(I-\gamma A^*A)^{\lfloor 1/\la\rfloor} A^*y_i$ with constant stepsize $\gamma=0.2$. For both Tikhonov regularization and Landweber iteration, we build a grid of regularization parameters $\Lambda=\{\lambda_1,\ldots, \lambda_N\}$ as in Assumption~\ref{aroughsnr}, namely with $\lambda_j=\lambda_1 Q^{j-1}$ for $j=1,\ldots,N$ and $Q=(\lambda_N/\lambda_1)^{1/(N-1)}$. In the case of Tikhonov, we choose $\Lambda\subseteq [10^{-4}, 100]$ with $N=500$ and so $Q\approx 1.0281$. For Landweber, we choose $\Lambda\subseteq[10^{-3}, 1]$, while $N$ remains the same and $Q\approx 1.0139$.  According to Section~\ref{sec:learn1param}, the regularization parameter learned by our approach is $\wh{\la}$, a minimizer of \eqref{eq:emp_err} on the grid $\Lambda$. In Figure~\ref{fig:filters}, the function $\la \in \Lambda \mapsto \widehat{L}(X_\la)$ is plotted for Tikhonov regularization. For Landweber, we plot the function in terms of number of iterations $k$.

\begin{figure}[h!]
	\label{fig:filters}
	\centering
	\includegraphics[width=\linewidth]{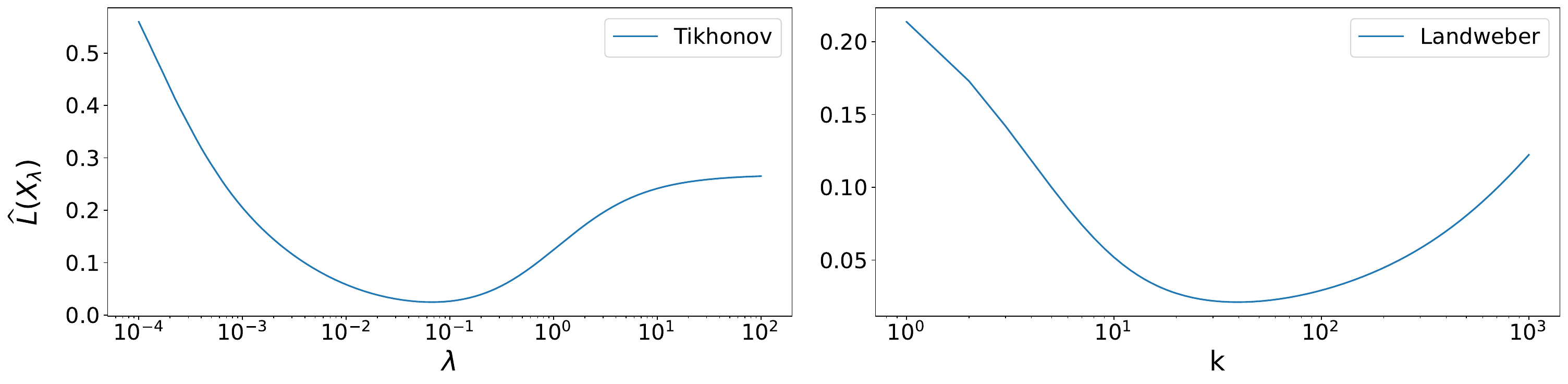}
	\caption{Empirical risk trajectories, of the Tikhonov and Landweber regularization methods, with respect to the regularization parameter $\la$.}
\end{figure}

\subsubsection{Illustration of Theorem~\ref{thm:apriori_spectral}}
In this section we investigate the dependence on the noise level $\tau$ of the error $L(X_{\la_*})$, see equation \eqref{final_bound} in Theorem~\ref{thm:apriori_spectral}. For every fixed noise level $\tau>0$ of $\varepsilon$, let $\la_*(\tau)$, or $k_*(\tau)$ in the case of Landweber, be a minimizer of the expected risk,
\be\label{eq:lastar}
\la_*(\tau)\in\argmin_{\la\in(0, +\infty)} L(X_\la).
\ee
 As stated in Theorem \ref{thm:apriori_spectral}, $L(X_{\la_*(\tau)})$ goes to zero when $\tau$ vanishes. The parameter $\alpha$ in Assumption~\ref{lip_spectral_reg} plays an important role in the bound, since $L(X_{\la_*(\tau)})\lesssim \tau^{4\alpha/(2\alpha+1)}$. In particular, we expect $L(X_{\la_*(\tau)})$ to go to $0$ faster when $\alpha$ increases. For Tikhonov, $\alpha=\min\{1,s\}$ (since $1$ is the qualification parameter for Tikhonov regularization). For Landweber, instead, $\alpha=s$. The influence of $s$ on the decay rate of the reconstruction error is shown in Figure \ref{fig:thm31} for the values $s=0.5$ and $s=1$. To determine $\lambda_*(\tau)$, we first consider $30$ different values of the noise level $\tau$ within the interval $[10^{-4},10^{-1}]$. The selected smoothness parameters allow us to gain  insights into the behavior of the expected risk with respect to the deterministic rate obtained in Theorem~\ref{thm:apriori_spectral}. In Figure \ref{fig:thm31}, we illustrate the quantity $L(X_{\lambda^*(\tau)})/\tau^{(4s)/(2s+1)}$, where it can be seen that all the curves are bounded when $\tau$ goes to zero. We can also observe that the quantity of interest is not going to zero, therefore suggesting that the derived bounds are tight.\\
\begin{figure}[h] 
	\label{fig:thm31}
	\centering
	\includegraphics[width=\linewidth]{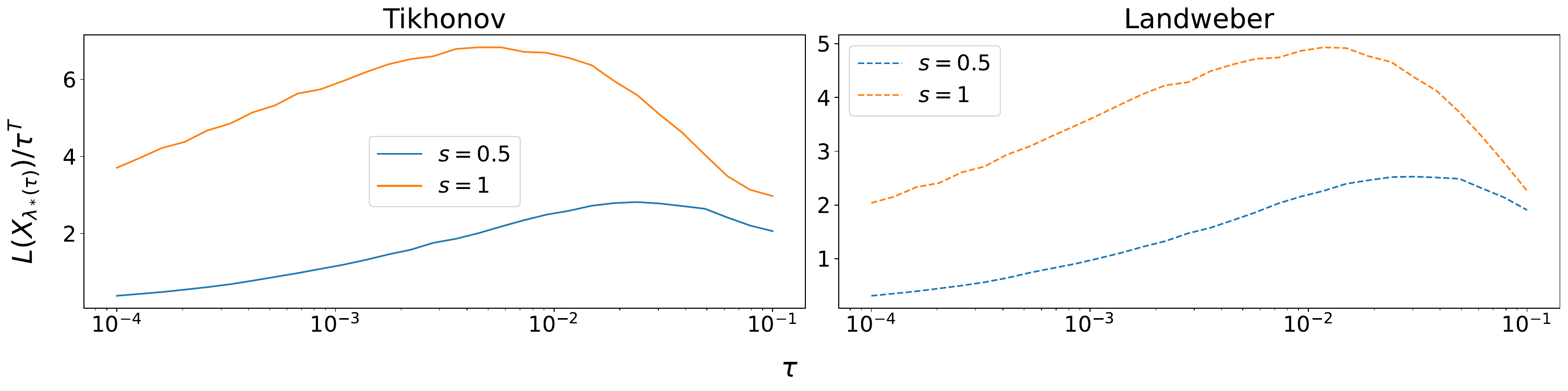}
	\caption{Behavior of $L(X_{\la_*})$ with respect to the rate $\tau^{T}$, $T=(4s)/(2s+1)$, obtained in Theorem \ref{thm:apriori_spectral}, for different smoothness parameters $s$ for Tikhonov and Landweber. It can be seen that each trajectory is upper bounded, as suggested by the rate in Theorem \ref{thm:apriori_spectral}. The horizontal axes are shown in logarithmic scale.}
\end{figure}

In the following experiment, we study the behavior of the best empirical regularization parameters, $\widehat{\la}(\tau)$ and $\widehat{k}(\tau)$, with respect to the noise level $\tau$ and the smoothness parameter $s$ for both Tikhonov and Landweber methods. Here, the empirical risk is computed with $10$ training points for smoothness parameters $s=0.5$ and $1$. We fix $30$ different values of the noise level $\tau$ in the interval $[10^{-4}, 10^{-1}]$, and we consider the following grids: $\Lambda\subseteq[10^{-5}, 1]$ with $N=500$ in the case of Tikhonov regularization, and $\Lambda\subseteq[10^{-4}, 1]$ with $N=5000$ for Landweber. It can be seen that the empirical parameters $\widehat{\la}(\tau)$ and $\widehat{k}(\tau)$ exhibit a similar behavior to the a priori optimal ones (\cite{engl} and Theorem~\ref{thm:apriori_spectral}): in the case of Tikhonov regularization, $\hat{\lambda}(\tau)$ increases with the noise,  and in the case of Landweber, the number of iterations decreases with respect to the noise. The smoothness parameter has also an effect on the optimal regularization parameter:
$\hat{\lambda}$ is increasing with respect to $s$, while the required  number of iterations in Landweber is decreasing. This behavior can  be observed in Figure \ref{fig:varynoise}.

\begin{figure}[h!]
	\label{fig:varynoise}
	\centering
	\includegraphics[width=\linewidth]{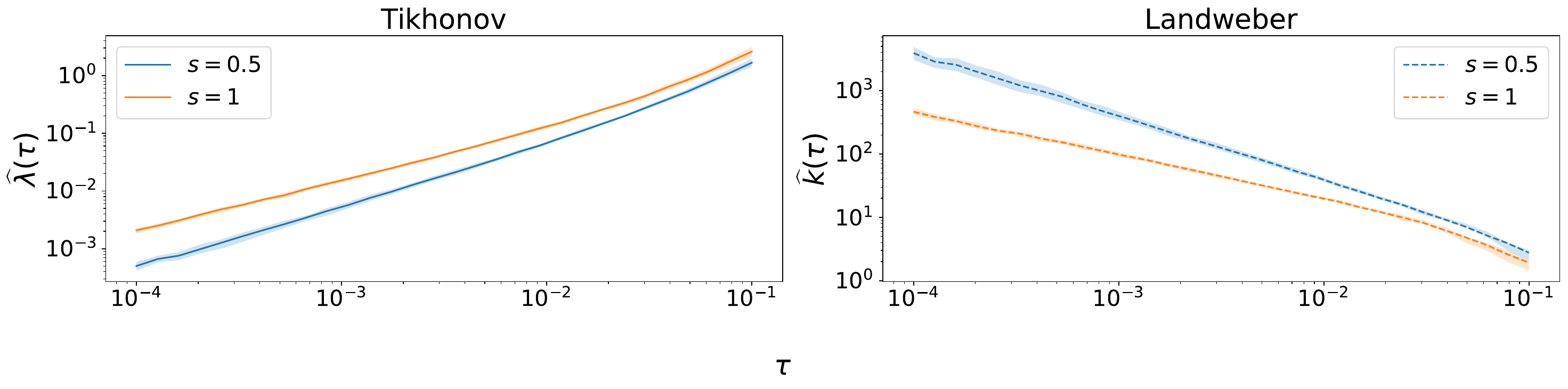}
	\caption{Value of $\wh{\la}$, $\wh k$ when varying the noise level for both Tikhonov and Landweber methods. Both parameters have been selected over a training set of $10$ points, constructed with different smoothness parameters as shown in the plot. Solid lines represent the mean value, while the shaded regions represent the $5^{\rm th}$-percentiles and $95^{\rm th}$-percentiles over 30 trials. Both axis are shown in logarithmic scale.}
\end{figure}

\subsubsection{Illustration of error bounds}\label{sec:exp_er}
In this section, we discuss some numerical experiments supporting the error bound stated in Corollary \ref{cor:spec}, both for Tikhonov and Landweber regularization methods. 
By Corollary \ref{cor:spec}, with high probability, there exist constants $c_2$, $c_3>0$ such that
$$
L(X_{\wh{\la}_\Lambda})\leq c_2\tau^{4\alpha/(2\alpha+1)}+ \frac{c_3}{n}.
$$
Therefore  $L(X_{\wh{\la}_\Lambda})$ with fixed noise level, behaves as $L(X_{\la_*})$ up to an additive constant. The same holds for fixed $n$, and $\tau\to 0$. 

We consider the same setting as for Figure \ref{fig:filters} with noise level $\tau=0.01$ and smoothness parameter $s=0.5$. We define the empirical risk, $\wh L (X_\la)$, for every $n\in\{5, 10,..., 100\}$, where we sample training points for every different value of $n$, and we denote by $\wh{\la}(n)$ and $\wh{k}(n)$ the parameters corresponding to the minimizers of the empirical risk with $n$ points. In Figure \ref{fig:excessrisk} we show that $L(X_{\wh{\la}(n)})$ goes to a certain constant, that depends on the noise level, when $n$ increases, see Figure~\ref{fig:excessrisk}.

\begin{figure}[h!]
	\label{fig:excessrisk}
	\centering
	\includegraphics[width=\linewidth]{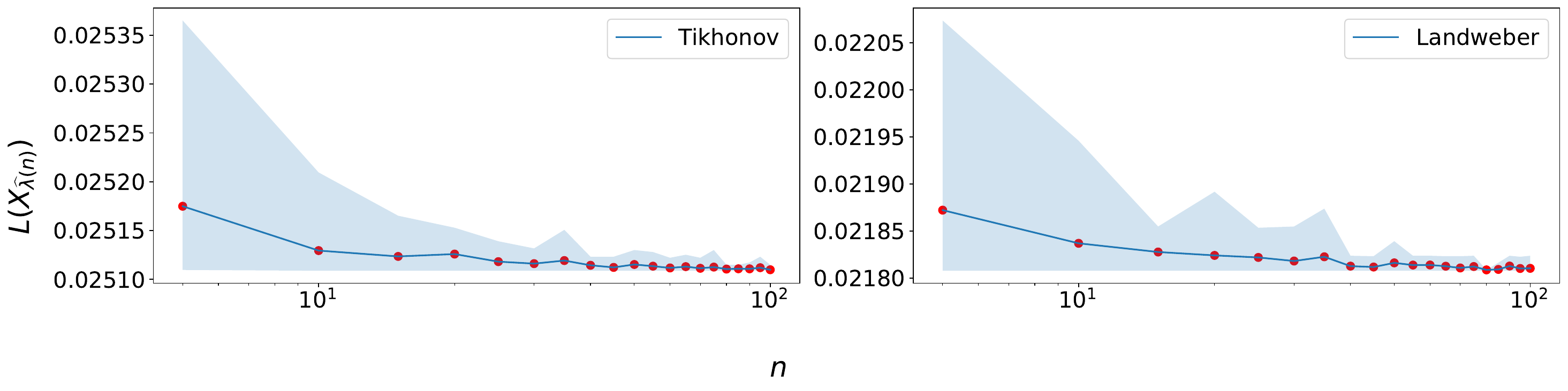}
	\caption{Behavior of $L(X_{\wh{\la}(n)})$, both for Tikhonov and Landweber regularization, as a function of $n$. The solid lines represent the mean value, while the shaded regions represent the $5^{\rm th}$-percentiles and $95^{\rm th}$-percentiles over 30 trials. The $x$-axis is shown in logarithmic scale. }
\end{figure}

Next, we illustrate the behavior of the expected risk $L$ with respect to the noise level $\tau$. First, we fix as smoothness parameter $s=0.5$ and consider $30$ different values of the noise level $\tau$ within the interval $[10^{-4}, 10^{-1}]$. Next, for every $\tau$, we find $\la_*(\tau)$, $k_*(\tau)$ as the minimizers of the expected risk $L$. Then, we fix the grid $\Lambda\subseteq[10^{-5}, 1]$ with $N=500$  and $Q \approx 1.0233$ in the case of Tikhonov, and $\Lambda\subseteq[10^{-4}, 1]$ with $N=3000$ and $Q \approx 1.0031$ in the case of Landweber. With this, we find $\wh{\la}_\Lambda(\tau)$, $\wh{k}_{\Lambda}(\tau)$ as the minimizers of the empirical risk $\wh{L}(X_\la)$, constructed with $n=5$  training points. In Figure \ref{fig:smallnbign} we plot, for every noise level $\tau$, the values $L(X_{\la_*(\tau)})$ and $L(X_{\wh{\la}_\Lambda(\tau)})$ in the cases of Tikhonov and Landweber, showing that their behavior with respect to $\tau$ is comparable.

\begin{figure}[h!]
	\label{fig:smallnbign}
	\centering
	\includegraphics[width=\linewidth]{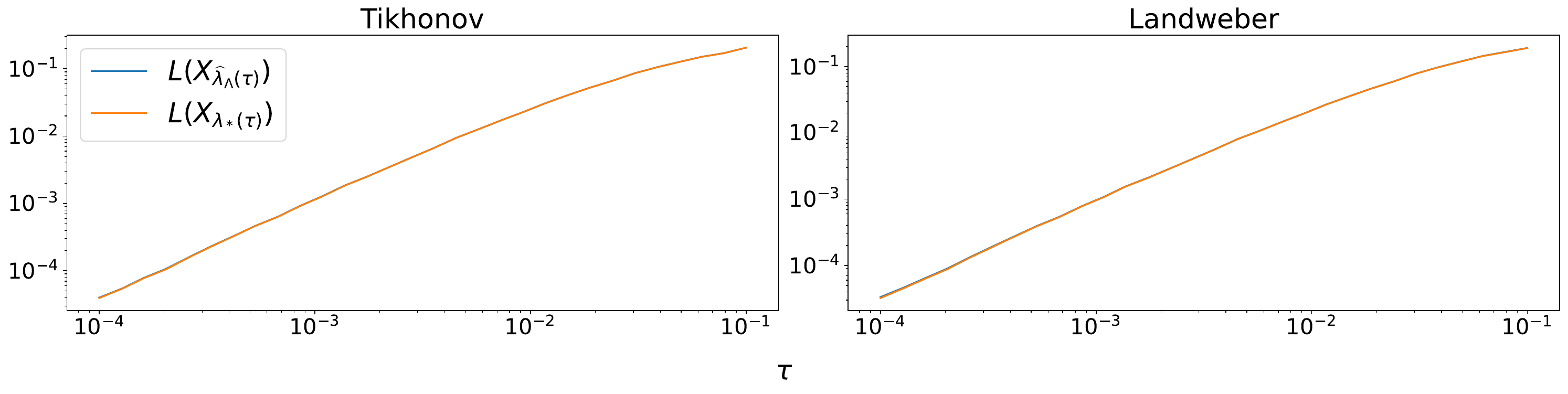}
	\caption{Comparison between $L(X_{\la_*(\tau)})$, in orange, and $L(X_{\wh{\la}_\Lambda(\tau)})$, in blue, when varying the noise level $\tau$ both for Tikhonov and Landweber regularization. As it can be observed, in such a scale, the lines almost coincide.}
\end{figure}

\subsubsection{Comparison with the quasi-optimality criterion} 
In this section we compare our data-driven approach to the quasi-optimality criterion \cite{tikhonov1979}. 
The latter is one of the most common and simple-to-implement heuristic parameter selection methods and does not require the noise level to be computed. Theoretical guarantees on its performance are available in the stochastic inverse problems setting~\cite{baure}. First, note that the computational cost of the two methods can be very different. The quasi-optimality criterion  performs instance-wise as all the usual parameter selection methods; i.e. given a set of test data $\{(y_i, x_i)\}_{i=1}^{n_\mathrm{test}}$, $n_\mathrm{test}\in\N$, and a regularization method $X_\la$, it outputs the best regularization parameter $\wh{\la}_i$ for each $y_i$, $i=1,...,n_\mathrm{test}$. This could lead to high computational costs  when the number of test points is big. Indeed, the method needs to be run as many times as the number of points, and for each test point the computation of the whole regularization path is required (see below). On the contrary, our algorithm requires to have access to a training set, but then, on test problems, the learned parameter $\wh{\la}$ will be the same for every $i=1,..., n_\mathrm{test}$, and only one regularized problem needs to be solved. In the following we compare the two approaches in terms of average performance on the test problems for Tikhonov and Landweber methods. \\

For Tikhonov regularization, we fix a grid of $N=1000$ regularization parameters $\Lambda\subseteq[10^{-5}, 10]$, with $Q\approx 1.0139$ and  denote by $X_{\la_j}^{i}$ the solution of the regularized problem for the parameter $\la_j$ and datum $y_i$, $i\in\{1,\ldots,n_{\mathrm{test}}\}$. We fix $n_{\mathrm{test}}=50$. 
For each pair $(y_i,x_i)$ in the test set, we select the parameter with the quasi-optimality criterion, namely we set
$\la^{\text{qo}}_i=\la_{j_*(i)}$, where $j_*(i)$ is defined as
$$
j_*(i)\in \argmin_{j\in\{1,\ldots,1000\}} \|X_{\la_j}^{i}-X_{\la_{j+1}}^{i}\|.
$$
Our method instead provides a unique $\wh{\la}_\Lambda$, depending on the training set. For this experiment, we fix a training set of $1000$ points. We then compare the average test error corresponding to the two methods, where, for the quasi-optimality criterion we consider
$$
L^{\text{qo}}=\frac{1}{50}\sum_{i=1}^{50}\|X_{\la^{\text{qo}}_i}^{i}-x_i\|^2.
$$
For the Landweber iteration,  we fix a grid of $N=800$ regularization parameters $\Lambda\subseteq[1/1000, 1]$, with $Q\approx 1.0087$ we follow the implementation of the quasi-optimality criterion proposed in~\cite{bauluk}, and we define $\la^{\text{qo}}_i=\la_{j_*}$, where $j_*(i)$ is defined as
$$
 j_*(i)\in \argmin_{j\in \{1,\ldots,800\}} \|X_{2\lfloor 1/\la_{j+1}\rfloor}^{i}-X_{\lfloor 1/\la_{j+1}\rfloor}^{i}\|,
$$
and we compare the average test error as for the Tikhonov method.

We denote the test error corresponding to our method $L^{\text{learn}}$ (for both Tikhonov and Landweber) and we compute the quantity $L^{\text{learn}}- L^{\text{qo}}$ for $30$ different realizations of the training set.  We show in tables \ref{tab:cv-qotik} and \ref{tab:cv-qoland} the mean value and standard deviation of the proposed experiment for both Tikhonov and Landweber with source condition $s=0.5$. As the tables suggest, the data-driven selection method performs differently than the quasi-optimality criterion for both Tikhonov and Landweber methods. On the one hand, in the case of Tikhonov regularization, the difference between the two studied methods is small when the noise level is small. Instead, when such noise level increases, the learned regularization parameter performs considerably better. In the case of Landweber, it can be seen in \ref{tab:cv-qoland} that the learned regularization parameter performs better for lower values of the noise level.
\begin{table}[]
	\label{tab:cv-qotik}
	\center
	\begin{tabular}{|ccccc|}
		\hline
		\multicolumn{5}{|c|}{$L^{\text{learn}}-L^{\text{qo}}$, Tikhonov}                                                                                      \\ \hline
		\multicolumn{1}{|l|}{noise lev.} & \multicolumn{1}{l|}{$\tau=10^{-3}$} & \multicolumn{1}{l|}{$\tau=10^{-2}$} & \multicolumn{1}{l|}{$\tau=10^{-1}$}  &   $\tau=0.5$  \\ \hline
		
		\multicolumn{1}{|c|}{mean}        & \multicolumn{1}{c|}{$-0.0025$}  & \multicolumn{1}{c|}{$-0.0665$}  & \multicolumn{1}{c|}{$-0.6071$}   & $-0.9935$ \\ \hline
		
		\multicolumn{1}{|c|}{std}         & \multicolumn{1}{c|}{$4.07\times 10^{-7}$}  & \multicolumn{1}{c|}{$4.27\times 10^{-6}$}  & \multicolumn{1}{c|}{$4.22\times 10^{-5}$}  & $0.0$ \\ \hline
	\end{tabular}
	\caption{Mean value and standard deviation of the error difference between our method and the quasi-optimality criterion. Above, we compare methods in the case of Tikhonov regularization for different values of the noise level.}
\end{table}

\begin{table}[]
	\label{tab:cv-qoland}
	\center
	\begin{tabular}{|ccccc|}
		\hline
		\multicolumn{5}{|c|}{$L_{\text{learn}}-L^{\text{qo}}$, Landweber}                                                                                      \\ \hline
		\multicolumn{1}{|l|}{noise lev.} & \multicolumn{1}{l|}{$\tau=10^{-3}$} & \multicolumn{1}{l|}{$\tau=10^{-2}$} & \multicolumn{1}{l|}{$\tau=10^{-1}$} & $\tau=0.5$  \\ \hline
		
		\multicolumn{1}{|c|}{mean}        & \multicolumn{1}{c|}{$-0.9987$}  & \multicolumn{1}{c|}{$-0.9348$}  & \multicolumn{1}{c|}{$-0.5042$}  &   $0.5775$ \\ \hline
		
		\multicolumn{1}{|c|}{std}         & \multicolumn{1}{c|}{$4.60\times 10^{-7}$}  & \multicolumn{1}{c|}{$1.56\times 10^{-6}$}  & \multicolumn{1}{c|}{$1.11\times 10^{-16}$}  & $0.0$ \\ \hline
	\end{tabular}
	\caption{Mean value and standard deviation of the error difference between our method and the quasi-optimality criterion with different values of the noise level.}
\end{table}

\subsection{Sparsity inducing regularizers}
In this section, we explore the theoretical results in Section \ref{sec:l1} for three different examples: denoising and deblurring of a sparse signal, and Total Variation regularization for image denoising. In particular, we will focus on illustrating, experimentally, Corollary \ref{cor:sparse}. To do so, we will perform the same experiments that we did for the spectral case in Section \ref{sec:exp_er}: first, we show that the expected risk for  $\wh{\lambda}_\Lambda$ with fixed noise level $\tau$, converges when the number of training points goes to infinity. Second, we show that the expected risk, when evaluated at the best empirical parameter $\wh{\lambda}_\Lambda$ for a fixed number of training points, has a comparable behavior to the expected risk evaluated at its minimum.\\

We start with the problem of denoising of a sparse signal.

\subsubsection{Denoising of a sparse signal}
Let $x^*\in\R^{d}$ be an $s$-sparse signal; i.e., a signal with $s$ nonzero entries, and consider the white noise model $\eps\sim N(0, \tau^2\mathrm{Id})$, with noise level $\tau>0$. We consider the denoising problem,
\begin{equation}\label{sigdenoi}
y=x^*+\eps,
\end{equation}
where $x^*$ is such that $\|x^*\|_2\leq 1$ as required. The most classical approach to recover $x^*$ having access only to $y$ is to solve the variational problem
\begin{equation}\label{eq:lasso}
\min_{x\in\mathbb{R}^d} \frac{1}{2}\|x-y\|^2_2+\la\|x\|_1.
\end{equation}
where the $\ell^1$ norm promotes sparsity~\cite{daudefdem}. In this case, it is easy to show that the solution admits a closed-form expression, that is
$$
X_\la = \mathcal{S}_\la(y), \quad \lambda\in(0, +\infty),
$$
where $\mathcal{S}_\la$ denotes the so-called soft-thresholding operator \cite{donoho95}, is defined componentwise as
$$
(\mathcal{S}_\la(y))_i:=
\begin{cases}
y_i-\la \mathrm{sign}(y_i), & \text{ if } |y_i|>\la,\\
\quad 0, & \text{ if } |y_i|\leq \la,

\end{cases}
$$
for every $i\leq d$.\\

As an illustrative example, we show in Figure \ref{fig:denotraj} the behavior of the empirical risk in this setting, where the training set $\{(y_i, x_i)\}_{i=1}^n$ is generated according to \eqref{sigdenoi} with $d=1024$, $s=64$ and noise level $\tau=0.25$,
$$
\wh L (\mathcal{S}_\la)=\frac{1}{n}\sum_{i=1}^{n} D_{\|\cdot\|_1}(x_i, \mathcal{S}_\la(y_i)),
$$
for $n=10$ and $\la$ in a grid $\La\subseteq[10^{-4}, 10]$ with $N=1000$ and so $Q\approx 1.0116$. It can be seen that this  behavior matches the  one predicted by the theoretical results. We first recall the Bregman divergence for this case, that is 
$$
D_{\|\cdot\|_1}(x, \mathcal{S}_\la(y)) = \|x\|_1-\langle s_{|\cdot|}(y), x\rangle=\|x\|_1-\langle \mathrm{sign}(y), x\rangle,
$$
for every $x$, $y\in\R^d$.
On the one hand, observe that, for every $i=1,..., n$, $\mathrm{sign}(\mathcal{S}_\la(x_i))\to \mathrm{sign}(y_i)$ as $\la\to 0$. This leads to
$$
\wh L(\mathcal{S}_\la)\to \frac1n\sum_{i=1}^n D_{\|\cdot\|_1}(x_i, y_i), \quad \text{as } \la\to 0,
$$
where the right hand side is constant. On the other hand, the for every $\la\in(0, +\infty)$ with $\la>\sup_{i=1,..., n} \|y_i\|_\infty$, we have that $\mathcal{S}_\la(y_i)=0$ for every $i=1,..., n$. Therefore, $D_{\|\cdot\|_1}(x_i, \mathcal{S}_\la(y_i))=\|x_i\|$ for every $i=1,..., n$ and so
$$
\wh L(\mathcal{S}_\la)\to \frac1n\sum_{i=1}^n \|x_i\|_1, \quad \text{as } \la\to +\infty,
$$
where the right hand side is again constant in this case. \\
\begin{figure}[h!]
	\centering
	\includegraphics[scale=.6]{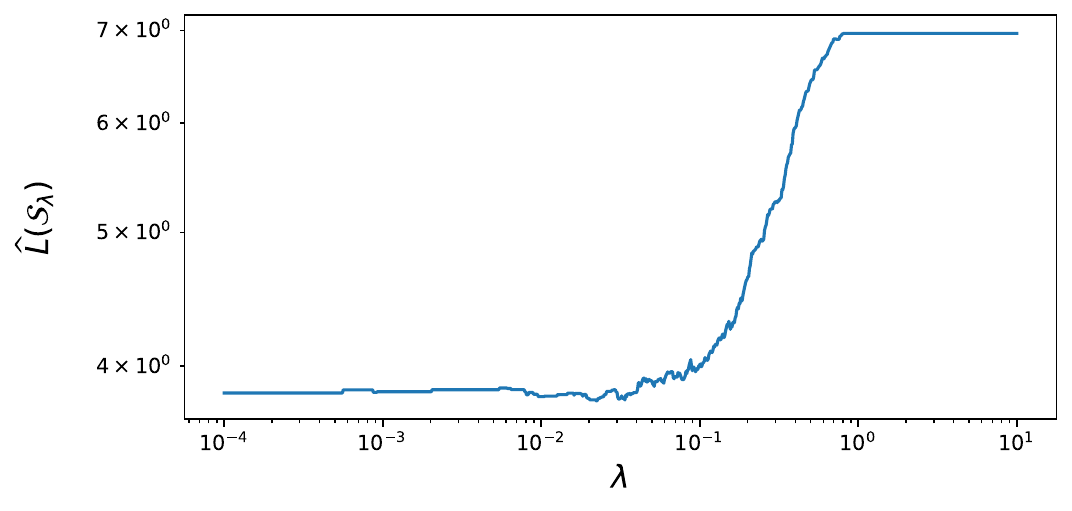}
	\caption{Behavior of $\wh L (\mathcal{S}_\la)$ with respect to the regularizaton parameter $\la$ for the signal denoising problem.}
 	\label{fig:denotraj}
\end{figure}

Next, we illustrate numerically Corollary \ref{cor:sparse} for this setting. First, we  show that the expected risk for the learned regularization parameter $\wh \la_\La$, converges  as $n$ goes to infinity. First, we fix as noise level $\tau=0.25$ and a grid of regularization parameters of $N=1000$ points $\Lambda\subseteq[10^{-5}, 1]$, with $Q\approx 1.0116$ and, for every $n\in\{1,2,...,20\}$ we define $\wh{\la}(n)$ as a minimizer of the empirical risk,
$$
\wh{\la}(n)\in\argmin_{\la\in\Lambda}\displaystyle\frac{1}{n}\sum_{i=1}^{n} D_{\|\cdot\|_1}(x_i, \mathcal{S}_\la(y_i)).
$$
where, for every $n$, we consider an independent set of training points $\{(y_i, x_i)\}_{i=1}^n$, generated according to \eqref{sigdenoi}. In Figure \ref{fig:excessl1}, we plot the quantity $L(\mathcal{S}_{\wh \la(n)})$ for different values of the dimension, $d=512$, $1024$ and $2048$ and fixed sparsity $s=16$, showing that it converges  when the number of training points goes to infinity, and the convergence does not depend on the dimension of the problem.\\

\begin{figure}[h!]
	\label{fig:excessl1}
	\centering
	\includegraphics[width=\linewidth]{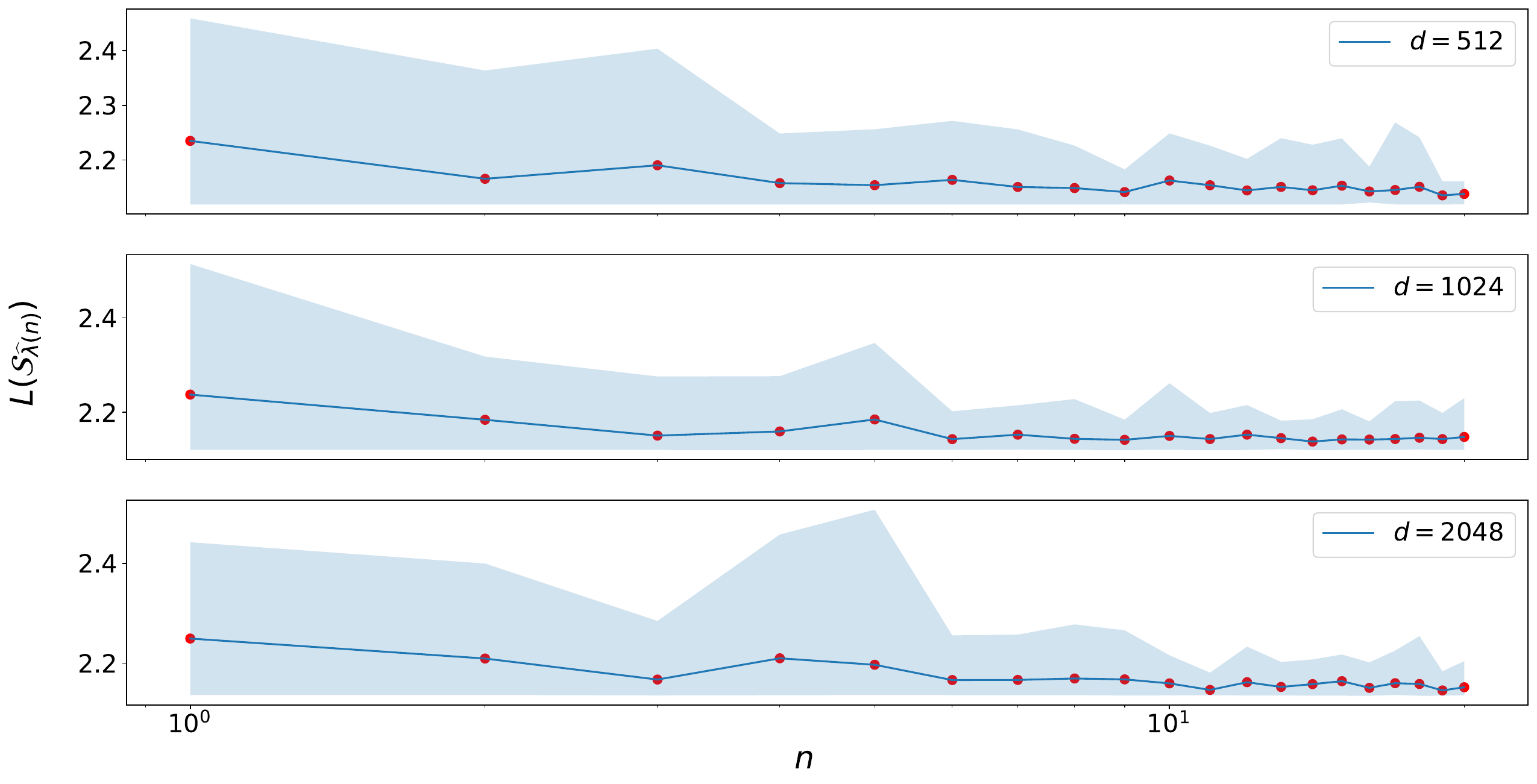}
	\caption{Behavior of $L(\mathcal{S}_{\wh \la (n)})$ as a function of $n$ for different values of the dimension. The solid lines represent the mean value, while the shaded regions represent  the values between the $5^{\rm th}$-percentiles and $95^{\rm th}$-percentiles over 30 trials. The $x$-axis is shown in logarithmic scale.}
\end{figure}

Finally, we show the behavior of the expected risk $L$ with respect to the noise level $\tau$. First, we fix $d=1024$ and sparsity $s=16$. Next, we fix $30$ different values of the noise level $\tau\in[0.1, 1]$. For every value of the noise level $\tau$, we find $\la_*(\tau)$ as the minimizer of the expected risk $L$. After, we consider the grid $\Lambda\subseteq[10^{-5}, 1]$ with $N=500$ and $Q\approx 1.0233$. Then, we find $\wh{\la}_\La$ as the minimizer of the empirical risk $\wh{L}$, constructed with $n=5$  training points. In Figure \ref{fig:Qdeno}, we show that the behavior of both $L(\mathcal{S}_{\la_*(\tau)})$ and $L(\mathcal{S}_{\wh{\la}_\Lambda(\tau)})$ with respect to $\tau$ is comparable. 

\begin{figure}[h!]
	\label{fig:Qdeno}
	\centering
	\includegraphics[width=\linewidth]{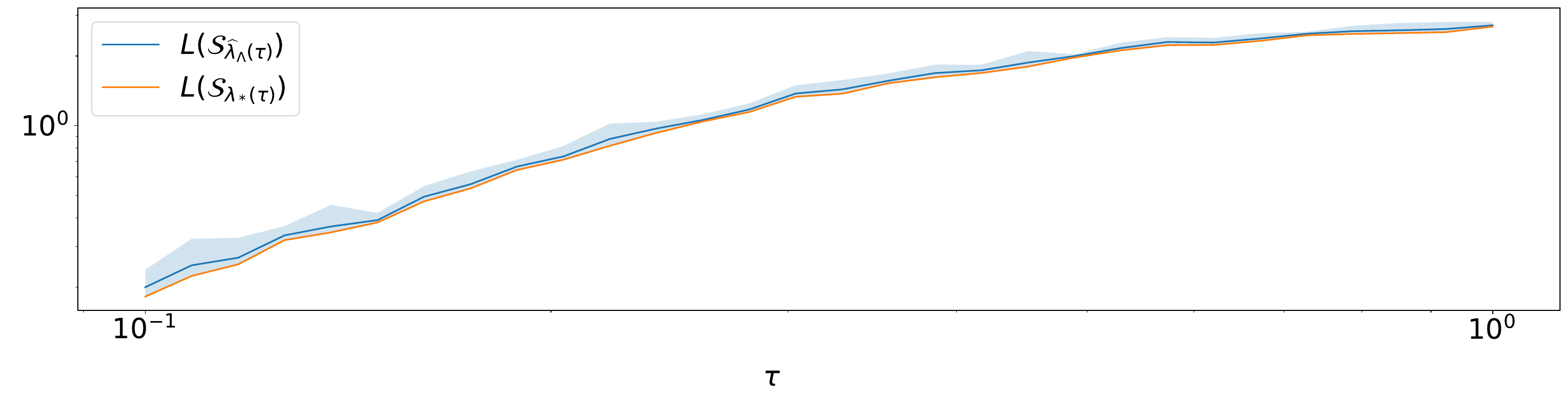}
	\caption{Behavior of the expected risk $L$ with respect to the noise level $\tau$. Recall that $L(\mathcal{S}_{\wh{\la}_\Lambda(\tau)})$ has been computed $30$ times. We therefore report the mean value, in a solid line, and the values between the $5^{\rm th}$-percentile and $95^{\rm th}$-percentile, which corresponds to the shaded region. Both axis are shown in logarithmic scale.}
\end{figure}

\subsubsection{Deblurring of a sparse signal}\label{sec:deblurring}
In this section, we consider the deblurring of a sparse signal\footnote{see \url{https://www.numerical-tours.com/python/}}. Our aim is to recover a sparse signal $x^*\in \mathbb{R}^{256}$ that has been corrupted via a convolution operator $A$ and additive noise,
\be\label{eq:deblurr}
y=Ax^*+\eps,
\ee
where $x^*$ is an $8$-sparse signal such that $\|x^*\|_2\leq 1$ by assumption, and $\eps\sim N(0, \tau^2\mathrm{Id})$ as pointed in Assumption \ref{lip_noise}. Moreover, the forward mapping $A$ is a linear convolution operator
$$
x\in\mathbb{R}^{256}\mapsto Ax=h\ast x \in \mathbb{R}^{256}, 
$$
with $h$ the second derivative of a Gaussian. More precisely, let $\phi(x)= e^{-x^2/(2\pi^2)}$, then $h=\phi''-\mu(\phi'')$, being $\mu(\phi'')$ the expectation of $\phi''$. In order to recover $x^*$, we solve the Lasso problem \cite{tibshi96}: 
\begin{equation}\label{eq:lassoreg}
\min_x \frac{1}{2}\|Ax-y\|^2_2+\la\|x\|_1,
\end{equation}
where $\la>0$ is the regularization parameter, running FISTA with constant stepsize \cite{fista} until convergence; i.e. until the difference between iterates is smaller than $10^{-6}$.\\


We now aim at illustrating Corollary \ref{cor:sparse}; i.e., showing the error behavior of the learned regularization parameter when $n$ goes to infinity. For this example, we fix $\tau = 0.1$ and the grid of admissible regularization parameters to be $\Lambda\subseteq[10^{-2}, 1]$ with $N=50$ and $Q\approx 1.0985$. The ERM in this case can be written as,
$$
\wh{\la}(n)\in\argmin_{\la\in\Lambda}\displaystyle\frac{1}{n}\sum_{i=1}^{n} D_{\|\cdot\|_1}(x_i, X^i_\la).
$$
where $X_\la^i:= X_\la(y_i)$ and, for every $n\in\{5,10,...,50\}$, we consider independent sets of training points $\{(y_i, x_i)\}_{i=1}^n$,  generated according to \eqref{eq:deblurr}. From Corollary \ref{cor:sparse}, the expected risk evaluated at $\wh \la(n)$, $L(X_{\wh \la(n)})$, should converge to  constant when $n\to \infty$. We plot this behavior in Figure \ref{fig:ermdeblurr}.\\

\begin{figure}[h!]
	\label{fig:ermdeblurr}
	\centering
	\includegraphics[width=\linewidth]{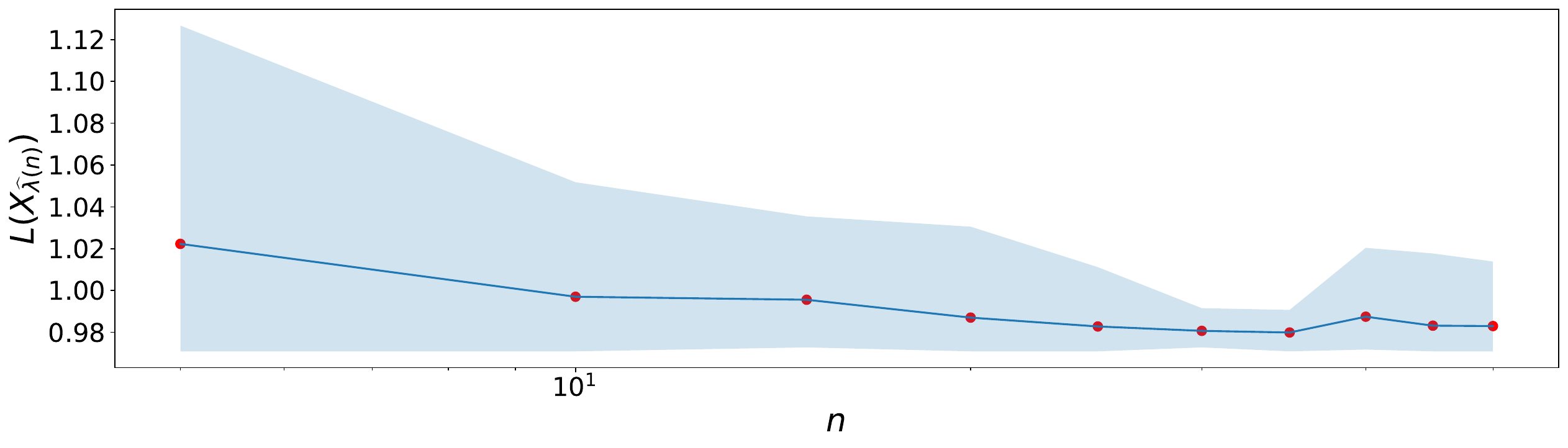}
	\caption{Behavior of $L(X_{\wh \la (n)})$ for the signal deblurring problem, showing that it goes to a certain constant as $n$ increases. The solid line represents the mean value, while the shaded region represents the value between $5^{\rm th}$-percentile and the $95^{\rm th}$-percentile over 30 trials. The $x$-axis is shown in logarithmic scale.}
\end{figure}

Next, we  show, empirically, that the behavior of the learned regularization parameter and the optimal one is comparable with respect to the noise level $\tau$. We  fix $30$ different values of the noise level within the interval $[0.1, 1]$ and define, for every $\tau$, $\la_*(\tau)$ as the minimizer of the expected risk $L$. After, we fix a grid of regularization parameters $\La\subseteq [10^{-2}, 1]$ with $N=10$ and $Q\sim 2.1544$. Hence, $\wh{\la}_\La$ will be the minimizer of the empirical risk $\wh{L}$, constructed with $n=5$  training points for every value of the noise level $\tau$. In Figure \ref{fig:Qdebl}, we plot the quantities $L(X_{\la_*(\tau)})$ and $L(X_{\wh{\la}_\Lambda(\tau)})$, showing that their behavior is comparable with respect to the noise level $\tau$.\\

\begin{figure}[h!]
	\label{fig:Qdebl}
	\centering
	\includegraphics[width=\linewidth]{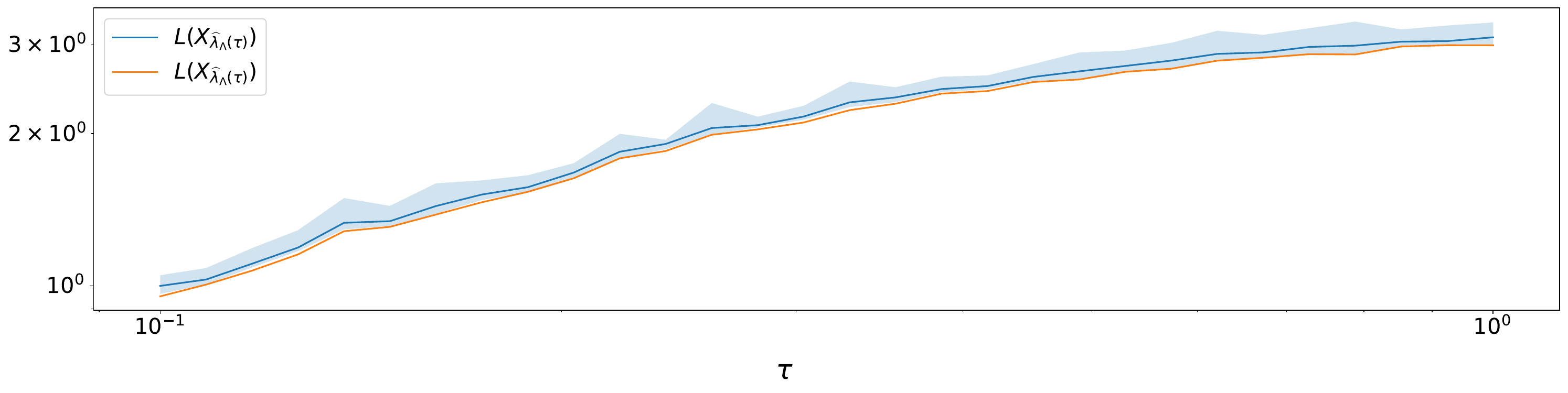}
	\caption{Behavior of the expected risk $L$, with respect to the noise level $\tau$, for the signal deblurring problem. Recall that $L(X_{\wh{\la}_\Lambda(\tau)})$, in blue has been computed $30$ times. We therefore report the mean value, in a solid line, and the values between the $5^{\rm th}$-percentile and $95^{\rm th}$-percentile, which corresponds to the shaded region. Both axis are shown in logarithmic scale.}
\end{figure}

Finally, we show one example of a reconstructed signal using our regularization parameter choice. In order to learn the parameter $\wh{\la}$, we first construct a training set of $n_{\mathrm{train}}=100$ clean and corrupted signals with the same distribution as the test element that we want to reconstruct, with noise level $\tau=0.1$. Then, the learned regularization parameter  minimizing  the empirical risk \eqref{eq:ermsparse}. We show in the third row of Figure~\ref{fig:deblurr}, the resulting regularized solution with the learned regularization parameter. 

\begin{figure}[h!]
	\label{fig:deblurr}
	\centering
	\includegraphics[width=\linewidth]{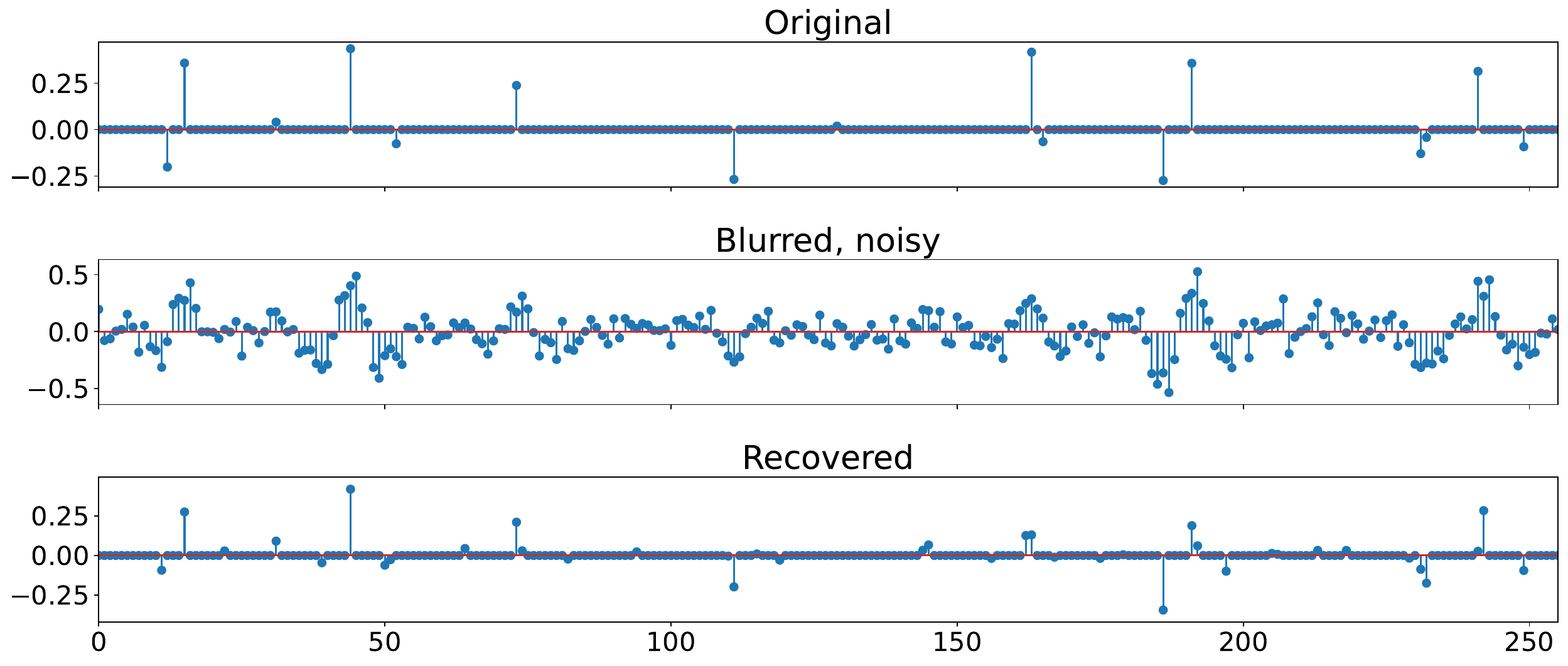}
	\caption{Deblurring of a sparse noisy, blurred signal with learned regularization parameter. In the first row, we show the original signal; in the second, its blurred and noisy version; and in the third row, the regularized solution with learned regularization parameter.}
\end{figure}

\subsubsection{Total Variation for image denoising}
In this section, we use our data-driven algorithm for choosing the regularization parameter of a Total Variation regularizer \cite{chambtv,rof92}. To do so, we focus on the image denoising problem 
\begin{equation}\label{eq:id}
y=x^*+\eps.
\end{equation}
where $x$, $y\in\R^{d\times d}$, $d\in\mathbb{N}$, and $\varepsilon\sim N(0, \tau^2\mathrm{Id})$ with noise level $\tau>0$. A classical approach to solve \eqref{eq:id} is to rely on the following variational approach \cite{scherzer09}
\begin{equation}\label{eq:tvden}
\min_x \frac12\|x-y\|_2^2+\lambda \mathrm{TV}(x),
\end{equation}
where $\la>0$ is the regularization parameter and
$$
\mathrm{TV}(x) = \|Dx\|_1,
$$
and $Dx=(D_1x, D_2x)\in \R^{2d(d-1)}$ is the discrete derivative, defined as in \cite{chamb2004}. Then, we propose as regularization method $X_\la$ a solution of problem~\eqref{eq:tvden}. 
Since \eqref{eq:tvden} does not have a closed-form solution, we compute it by running FISTA on the dual problem of \eqref{eq:tvden}, until convergence, i.e. until the difference between iterates is smaller than $10^{-8}$. To illustrate Corollary \ref{cor:sparse}, we first show the behavior of the expected risk, evaluated at the learned regularization parameter $\wh \la_\La$ for this example.

We consider the MNIST dataset \cite{mnist} of $28\times 28$ images of digits from $0$ to $9$, and corrupt them as in \eqref{eq:id}. To illustrate Corollary \ref{cor:sparse}, we fix the noise level $\tau=0.25$. Then, we fix a grid of $N=50$ points $\Lambda\subseteq[10^{-3}, 1]$, with $Q\approx 1.1514$. For every $n\in\{5,10,...,50\}$, we let $\wh{\la}(n)$ be a minimizer of the empirical risk,
$$
\wh{\la}(n)\in\argmin_{\la\in\Lambda}\displaystyle\frac{1}{n}\sum_{i=1}^{n} D_{\mathrm{TV}}(x_i, X^i_\la).
$$
where, for every $n$, we consider an independent training set of points $\{(y_i, x_i)\}_{i=1}^n$ randomly selected from a set of $3000$ images. We therefore plot in Figure \ref{fig:ermTV} the behavior of the expected risk $L$ for $\wh{\lambda}(n)$, and show it converges  as $n\to\infty$.\\

\begin{figure}[h!]
	\label{fig:ermTV}
	\centering
	\includegraphics[width=\linewidth]{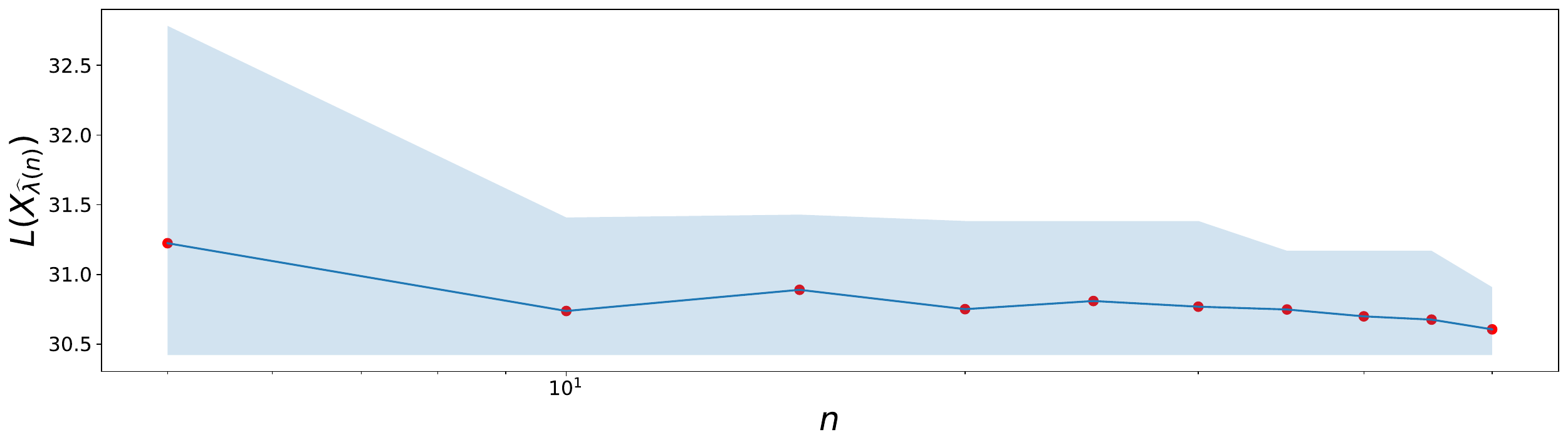}
	\caption{Behavior of $L(X_{\wh \la (n)})$ as a function of $n$ for image denoising problem, showing that it goes to a certain constant as $n$ increases. The solid line represents the mean value, while the shaded region represents the value between the $5^{\rm th}$-percentile and $95^{\rm th}$-percentile over 30 trials. The $x$-axis is shown in logarithmic scale.}
\end{figure} 

Next, we  illustrate the behavior of the expected risk. We consider the  same experimental setting as we did for Figure \ref{fig:Qdebl} for the signal deblurring problem, and we show in Figure \ref{fig:QTV} that the behavior of both $L(X_{\la_*(\tau)})$ and $L(X_{\wh{\la}_\Lambda(\tau)})$ is comparable with respect to $\tau$. \\

\begin{figure}[h!]
	\label{fig:QTV}
	\centering
	\includegraphics[width=\linewidth]{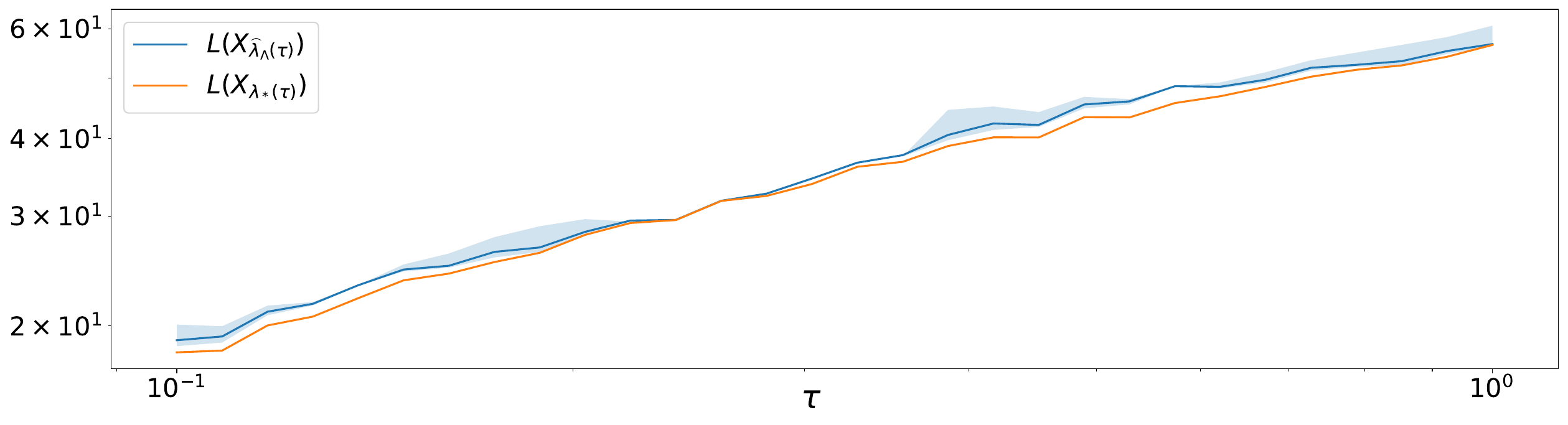}
	\caption{Comparison between $L(X_{\la_*(\tau)})$, in orange, and $L(X_{\wh{\la}_\Lambda(\tau)})$, in blue, when varying the noise level $\tau$ for the Total Variation regularization. In the case of $L(X_{\wh{\la}_\Lambda(\tau)})$, the solid line represents the mean value, while the shaded region represents the values between the $5^{\rm th}$-percentile and $95^{\rm th}$-percentile. Bot axes are shown in logarithmic scale.}
\end{figure}

Finally, as an illustrative example, we explore the performance of our parameter selection method on test images from the MNIST dataset. We compute four different data-driven regularization parameters for four different training sets, each of $100$ training points, and check the reconstruction results of the TV regularized solution for two different digits in the test set. The results are shown in Figure \ref{fig:tvdenoising}. We observe that the recovery results on single test images may vary depending on the set of points that was used for training. This is expected, since our parameter selection method has been designed in order to perform effectively on average. 
\begin{figure}[h!]
	\label{fig:tvdenoising}
	\centering
	\includegraphics[width=\linewidth]{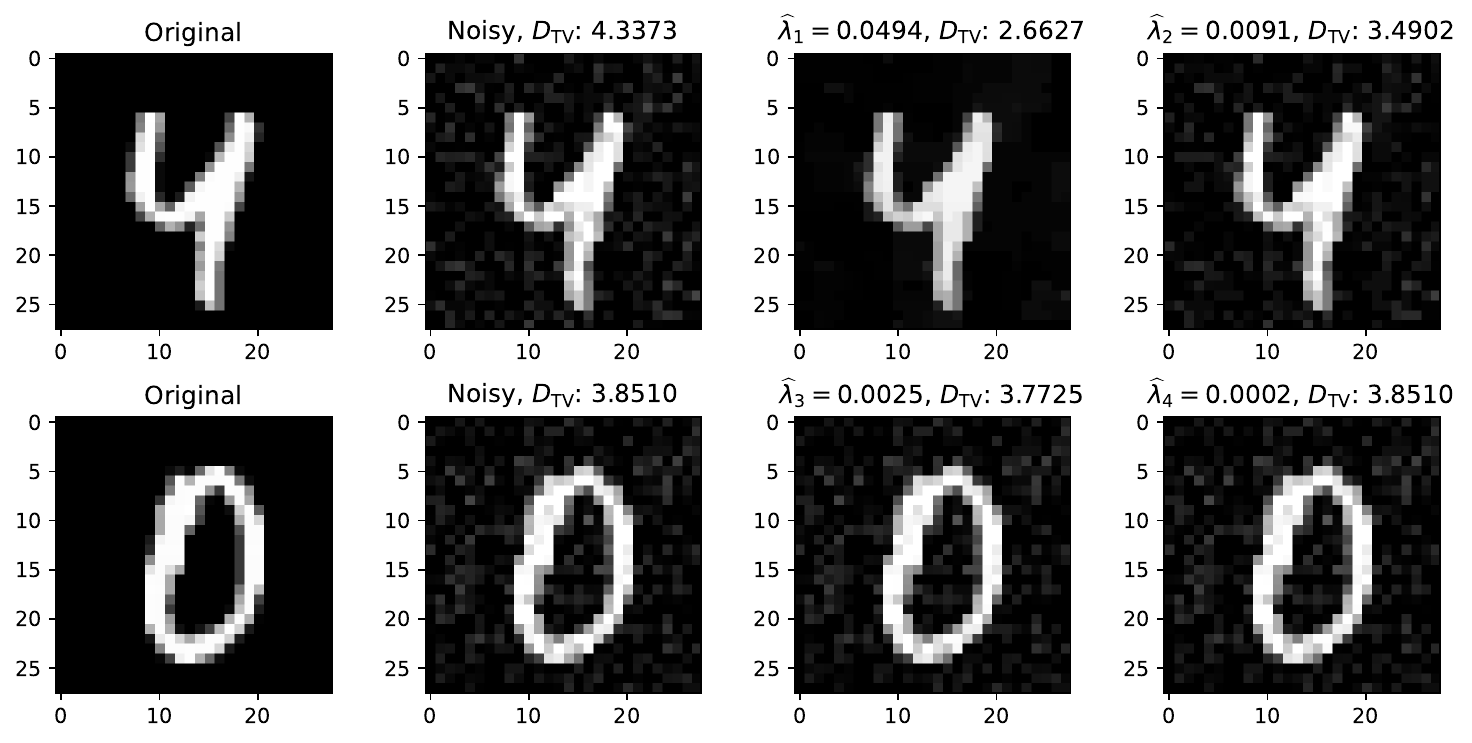}
	\caption{Total Variation denoising algorithm for two digit in the test set. From left to right, in every row, we plot the original image, its noisy version, and the recovery obtained with different regularization parameters. We also include, accordingly, the Bregman divergence with respect to the original image and the value of the regularization parameter that has been used for such recovery.}
\end{figure}

\section{Conclusions}
We studied the problem of learning the regularization parameter in statistical inverse problems. In particular we consider a data driven approach that we cast as 
an instance of empirical risk minimization, common in machine learning. Borrowing results from statistical learning theory, we derive general error guarantees, that we specialize considering different classes of inverse problems and regularization methods. Theoretical results are illustrated by extensive numerical experiments to illustrate.\\ 

Possible developments include considering higher dimensional paramterization of the regularization approaches, as well as possible instance dependent parameter choices. We think our analysis provides useful insights into how to tackle these questions. 

\section{Acknowledgements}
This project has been supported by the TraDE-OPT project, which received funding from the European Union’s Horizon 2020 research and innovation program under the Marie Skłodowska-Curie grant agreement No 861137. L. R. acknowledges the Center for Brains, Minds and Machines (CBMM), funded by NSF STC award CCF-1231216. J.C.R and L.R. acknowledge the Ministry of Education, University and Research (grant ML4IP R205T7J2KP). L.R. and S. V. acknowledge the European Research Council (grant SLING 819789), the US Air Force Office of Scientific Research (FA8655-22-1-7034). The research by E.D.V., L. R., C.M. and S.V. has been supported by the MIUR grant PRIN 202244A7YL. The research by E.D.V., C.M. and S.V. has been supported by the MIUR Excellence Department Project awarded to Dipartimento di Matematica, Universita di Genova, CUP D33C23001110001. E.D.V., C.M. and S.V. are members of the Gruppo Nazionale per l’Analisi Matematica, la Probabilità e le loro Applicazioni (GNAMPA) of the Istituto Nazionale di Alta Matematica (INdAM). This work represents only the view of the authors. The European Commission and the other organizations are not responsible for any use that may be made of the information it contains.
 

\bibliographystyle{siamplain}
\bibliography{references}

\end{document}